\pdfoutput=1

\documentclass[12pt]{amsart}
\usepackage[dvips]{graphicx}
\usepackage{color}
\usepackage{amssymb} 
\usepackage[english]{babel}
\usepackage{amsmath}
\usepackage{amsfonts}
\usepackage{latexsym}
\usepackage{amscd}
\usepackage[latin1]{inputenc}
\usepackage{enumerate}
\usepackage{afterpage}

\def\a{\alpha}
\def\b{\beta}

\def\e{\epsilon}
\def\g{\gamma}
\def\i{{\bf i}}

\def\s{\sigma}

\def\cal{\mathcal}
\def\CC{\mathbb C}
\def\Chat{\hat {\mathbb C}}

\def\ZZ{\mathbb Z}

 \def\l{{\lambda}}
\def\M{{\cal M}}
\def\S{{\cal S}}
\def\SC{{\mathcal {SC}}}

\def\M{{\cal M}}
\def\P{{\cal P}}
\def\PC{\mathcal {PC}}
\def\RR{{\cal R}}
\def\T{{\cal T}}

\newcommand{\C}{\mathbb{C}}

\newcommand{\R}{\mathbb{R}}
\newcommand{\NN}{\mathbb{N}}
\newcommand{\Z}{\mathbb{Z}}

\newcommand{\D}{\mathbb{D}}
\newcommand{\HH}{\mathbb{H}}
\def\tr{\mathop{\rm  Tr}}

\def\Id{\mathop{\rm  Id}}
\renewcommand{\to}{\longrightarrow}

\def\QF{{\cal {QF}}}

\def\Tw{\mathop { {Tw}}} 
\def\teich{{\cal T}}
\def\Teich{Teichm\"uller }
\def\dd{\partial}
 
\def\ttau { {\underline { \tau }}}
\def\tw{{tw}}

\def\varnothing{\emptyset}
\newtheorem{Theorem}{Theorem}[section]

\newtheorem{Lemma}[Theorem]{Lemma}
\newtheorem{Proposition}[Theorem]{Proposition}

\newtheorem{Remark}[Theorem]{Remark}

\newtheorem{introthm}{Theorem}

\begin{document}

\title[Top terms of polynomial traces] {Top terms of polynomial traces in Kra's plumbing construction.}

\begin{abstract}
 
Let $\Sigma$ be a surface of negative Euler characteristic together with a pants decomposition $\P$. Kra's plumbing construction endows $\Sigma$ with a projective structure as follows. Replace each pair of pants by a triply punctured sphere and glue, or `plumb', adjacent pants by gluing punctured disk neighbourhoods of the punctures. The gluing  across the $i^{th}$ pants curve  is defined by a complex parameter $\tau_i \in \CC$. The associated holonomy representation $\rho: \pi_1(\Sigma) \to PSL(2,\C)$  gives a projective structure on $\Sigma$ which depends holomorphically on the $\tau_i$. In particular, the traces of all elements $\rho(\g) , \g \in \pi_1(\Sigma)$, are polynomials in the $\tau_i$.

Generalising results proved in~\cite{kstop, seriesmaskit} for the once and twice punctured torus respectively, we prove a formula giving a simple linear relationship between the coefficients of the top terms of $\rho(\g)$, as polynomials in the $\tau_i$, and the Dehn-Thurston coordinates of $\g$ relative to $\P$.

This will be applied elsewhere to give a formula for the asymptotic directions of pleating rays in the Maskit embedding of $\Sigma$ as the bending measure tends to zero, see~\cite{maloni}.

\medskip

\end{abstract}

\author{Sara Maloni and Caroline Series} 

\address{\begin{flushleft}\rm {\texttt{S.Maloni@warwick.ac.uk \; (http://www.maths.warwick.ac.uk/$\sim$marhal),\\ 
C.M.Series@warwick.ac.uk \;  (http://www.maths.warwick.ac.uk/$\sim$masbb), } } \\ Mathematics Institute, 
 University of Warwick \\ Coventry CV4 7AL, UK \end{flushleft}}
\date{\today}
\maketitle

\section{Introduction}
\label{sec:introduction}

Let $\Sigma$ be a surface of negative Euler characteristic together  with a  pants decomposition $\P$. Kra's plumbing construction endows $\Sigma$ with a projective structure as follows. Replace each pair of pants by a triply punctured sphere and glue, or `plumb', adjacent pants by gluing punctured disk neighbourhoods of the punctures. The gluing  across the $i^{th}$ pants curve  is defined by a complex parameter $\tau_i \in \CC$. More precisely,  $zw = \tau_i$ where $z,w$ are standard holomophic coordinates on punctured disk neighbourhoods of the two punctures. The associated holonomy representation $\rho: \pi_1(\Sigma) \to PSL(2,\C)$  gives a projective structure on $\Sigma$ which depends holomorphically on the $\tau_i$, and in which the pants curves themselves are automatically parabolic. In particular, the traces of all elements $\rho(\g) , \g \in \pi_1(\Sigma)$, are polynomials in the $\tau_i$.

The main result of this paper is a very   simple relationship between the coefficients of the top terms of $\rho(\g) $, as polynomials in the $\tau_i$, and the Dehn-Thurston coordinates of $\g$ relative to $\P$. This generalises results of~\cite{kstop, seriesmaskit} for the once and the twice punctured torus respectively.  

Our formula is as follows. Let $\S$ denote the set of homotopy classes of multiple loops on $\Sigma$, and let the pants curves defining $\P$ be $\s_i, i=1,\ldots,\xi$. (For brevity we usually refer to elements of $\S$ as \emph{curves}, see Section~\ref{sec:background}.) The \emph{Dehn-Thurston coordinates} of $\g \in \S$ are $\i(\g) = (q_i,p_i), i=1,\ldots,\xi$,  where $q_{i} = i(\g,\s_{i}) \in \NN \cup \{0\}$ is the geometric intersection number between $\g$ and  $\s_{i}$ and $p_{i} \in \ZZ$ is the twist of $\g$ about $\s_i$. We prove:

\begin{introthm}\label{thm:traceformula}  
Let $\gamma$ be a connected simple closed curve on $\Sigma$, not parallel to any of the pants curves $\s_i$. Then $\tr \rho(\g)$ is a polynomial in $\tau_{1}, \cdots , \tau_{\xi}$ whose top terms are given by:
\begin{align*} 
\tr \rho(\g) &=\pm i^{q} 2^{h} \Bigl(\tau_{1}+\frac{(p_{1}-q_{1})}{q_{1}}\Bigr)^{q_{1}} \cdots \Bigl(\tau_{\xi}+\frac{(p_{\xi}-q_{\xi})}{q_{\xi}}\Bigr)^{q_{\xi}} + R,\\  
& = \pm i^{q} 2^{h} \left(\tau_{1}^{q_{1}}\cdots\tau_{\xi}^{q_{\xi}} +\sum_{i=1}^{\xi} (p_{i}-q_{i}) \tau_{1}^{q_1} \cdots \tau_i ^{q_{i}-1}\cdots\tau_{\xi}^{q_{\xi}} \right) + R     
\end{align*} 
where 
\begin{itemize}
	\item $q =\sum_{i=1}^{\xi}q_{i}>0$;
	\item $R$ represents terms with total degree in $\tau_{1} \cdots \tau_{\xi}$ at most $q - 2$ and of degree at most $q_{i}$ in the variable $\tau_{i}$;
	\item $h= h(\g)$ is the total number of $scc$-arcs in the standard representation of $\g$ relative to $\P$, see below. 
\end{itemize}
If $q =0$, then $\gamma = \s_i$ for some $i$,
$\rho(\g)$ is parabolic,  and $\tr \rho(\g) = \pm 2$.
\end{introthm}

The non-negative integer $h = h(\g)$ is defined as follows. The curve $\g$ is first arranged to intersect each pants curve minimally. In this position, it intersects a pair of pants $P$ in  a number  of arcs joining boundary loops of  $P$. We call one of these an \emph{$scc$-arc}  (short for same-(boundary)-component-connector,  called an \emph{archetype} in Penner~\cite{Penner})  if it joins one boundary component to itself, and  denote by $h $  the total number of  $scc$-arcs, taken over all pants in $\P$.

The precise definition of the twist coordinates $p_i$ in Theorem~\ref{thm:traceformula} requires some care; we use essentially the standard definition implied in~\cite{flp} and explained in detail in~\cite{Dylan} (called here the DT-twist, see  Section~\ref{sec:dylantwist}), although for the proof we find useful the form given by Penner~\cite{Penner} (called here the P-twist and denoted $\hat p_i$, see  Section~\ref{sec:pennertwist}).

We remark that the formula in Theorem~\ref{thm:traceformula}  could of course be made neater by replacing the parameter $\tau$ by $\tau-1$; we use $\tau$ to be in accordance with the conventions of~\cite{kstop,seriesmaskit}, see also Section~\ref{sec:marking2}.
 
\medskip

We believe the formula in Theorem~\ref{thm:traceformula} noteworthy  in its own right. However the main motivation for this work  was the following. If the representation $\rho$ constructed in the above manner is free and  discrete, then the resulting hyperbolic $3$-manifold $M=\HH^3/\rho(\pi_1(\Sigma))$  lies on the boundary of quasifuchsian space $\QF(\Sigma)$. One end of $M$ consists of  a union of triply punctured spheres obtained by pinching in $\Sigma$ the curves $\s_i$ defining $\P$. Suppose that, in addition, $\rho( \pi_1(\Sigma))$ is geometrically finite and that the other end $\Omega^+/\rho( \pi_1(\Sigma))$ of $M$ is  a Riemann surface homeomorphic to $\Sigma$. Since the triply punctured spheres are rigid, it follows from Ahlfors-Bers' measurable Riemann mapping theorem that the Riemann surface structure of $\Omega^+/\rho( \pi_1(\Sigma))$ runs over the \Teich space $\T(\Sigma)$ of $\Sigma$. The image of the space of all such groups in the character variety $\cal R$ of $\Sigma$ is called the \emph{Maskit embedding} of $\T(\Sigma)$.

In~\cite{kstop, seriesmaskit}, special cases of the trace formula were important in  constructing a computational method of locating  the image $\cal M$ of $\T(\Sigma) $  in $\cal R$. In those papers we defined a \emph{pleating ray} to be a line in $\cal R$ along which the projective class of the bending measure was kept constant. The trace formulae enabled us to find the asymptotic directions of pleating rays in $\M$ as the bending measure tends to zero. Theorem~\ref{thm:traceformula} allows  the extension of these results to the general case, see~\cite{maloni}.
 
The plan of this paper is as follows. After establishing preliminaries in Section~\ref{sec:background}, in Section~\ref{sec:twist} we review the Dehn-Thurston coordinates and, in particular, the definition of twists. In Section~\ref{sec:gluing} we discuss the gluing construction which leads to the family of projective structures and their holonomy representation. In Section~\ref{sec:examples} we explain in detail the holonomy representation in various special cases, starting with arcs in a single pair of pants and going on to the  one holed torus and four holed sphere. Finally, in Section~\ref{sec:mainthm},  we make explicit the general combinatorial pattern of  matrix products obtained in the holonomy,   and use this to give an inductive proof of Theorem~\ref{thm:traceformula}.

\section{Background and Notation}
\label{sec:background}

Suppose given a surface $\Sigma = \Sigma_{g}^{b}$ of finite type and negative Euler characteristic, and choose a maximal set $\PC = \{ \s_{1}, \ldots, \s_{\xi}\}$ of homotopically distinct and non-boundary parallel loops in $\Sigma$ called \emph{pants curves}, where  $\xi = \xi(\Sigma) = 3g-3+b$ is the complexity of the surface. These connected curves split the surface into $k=2g-2+b$ three-holed spheres $ P_1,\ldots, P_k$, called \emph{pairs of pants}. (Note  that  the boundary of $P_i$ may  include  punctures of $\Sigma$.)  We refer to both the set $\cal P = \{ P_1,\ldots, P_k\}$, and the set $\PC $, as a \emph{pants decomposition} of  $\Sigma$.

We take $P_i$ to be a closed three-holed sphere whose interior $\mathrm{Int}(P_i)$ is embedded in $\Sigma$; the closure of $\mathrm{Int}(P_i)$ fails to be embedded precisely in the case in which two of its boundary curves are identified in $\Sigma$, forming an embedded one-holed torus $\Sigma_{1,1}$. Thus each pants curve $\s= \s_i$ is the common boundary of one or two pants whose union we refer to as the \emph{modular} surface associated to $\s$, denoted $M(\s)$. If the closure of $\mathrm{Int}(P_i)$ fails to be embedded then $M(\s)$ is a one-holed torus $\Sigma_{1,1}$, otherwise it is a four-holed sphere $\Sigma_{0,4}$.

Any hyperbolic pair of pants $P$ is made by gluing two right angled hexagons along three alternate edges which we call its \emph{seams}. In much of what follows, it will be convenient to designate one of these hexagons as `white' and one as `black'. A properly embedded arc  in $P$, that is, an arc with its endpoints on $\dd P$, is called  $scc$ (same component connector) if it has both its endpoints on the same  component of $\dd P$  and $dcc$ (different component connector) otherwise.

Let $\S_0 = \S_0(\Sigma)$ denote the set of free homotopy classes of connected closed simple non-boundary parallel loops on $\Sigma$, and let $\S = \S(\Sigma)$ be the set of  curves on $\Sigma$, that is, the set of finite unions of non-homotopic curves in $\S_0$. For simplicity we usually refer to elements of $\S$ as  `curves' rather than `multi-curves', in other words, a curve is not required to be connected. The geometric intersection number $i(\a,\b)$  between  $\a,\b \in \S$ is the least number of intersections between curves representing the two homotopy classes, that is $$i(\a, \b) = \min_{a \in \a , \; b \in \b} |a \cap b|.$$  

\subsubsection{Convention on dual curves}
\label{sec:dualcurves}

We shall need  to consider \emph{dual curves} to $\s_i \in \PC$, that is, curves which intersect $\s_i$ minimally and which are completely contained in $M(\s_i)$, the union of the pants $P,P'$ adjacent to $\s_i$. The intersection number of such a connected curve with $\s_i$ is $1$  if $M(\s_i)$ a  one-holed torus and $2$ otherwise.  In the first case, the curve is made by identifying the endpoints of a single $dcc$-arc  in the pair of pants adjacent to $\s_i$ and, in the second, it is the union of two $scc$-arcs, one in each of the two pants whose union is $M(\s_i)$. We adopt a useful convention introduced in~\cite{Dylan} which simplifies the formulae in such a way as to avoid  the need to distinguish  between these two cases. Namely, for those $\s_i$  for which $M(\s_i)$ is $\Sigma_{1,1}$, we define the dual curve $D_i \in \S$  to be \emph{two} parallel copies of the connected curve intersecting $\s_i$ once, while if $M(\s_i)$ is $\Sigma_{0,4}$ we take a single copy. In this way we always have, by definition, $i(\s_i, D_i) =2$, where $i(\a,\b)$ is the geometric intersection number as above. 

A \emph{marking} on $\Sigma$ is the specification of a fixed base surface $\Sigma_0$, together with a  homeomorphism $\Psi: \Sigma_0 \to \Sigma$. Markings can be defined in various  equivalent ways, for example by specifying the choice of dual curves,  see Section~\ref{sec:dylantwist} below.  
 
\subsubsection{Convention on twists}
\label{sec:dehntwists}

Our convention will always be to measure twists to the right as positive. We denote by ${\Tw}_{\s}(\g)$ the \emph{right} Dehn twist of the curve $\g$ about the curve $\s$. 

\section{Dehn-Thurston coordinates}
\label{sec:twist}

Suppose we are given a surface $\Sigma$ together with a pants decomposition $\cal P$ as above. Let $\g \in \S$ and for $i = 1, \ldots, \xi$, let $q_{i} = i(\g,\s_{i}) \in \Z_{\geqslant 0}$. Notice that if $\s_{i_{1}}, \s_{i_{2}}, \s_{i_{3}}$ are pants curves which together bound a pair of pants whose interior is embedded in $\Sigma$, then the sum $q_{i_{1}} + q_{i_{2}} + q_{i_{3}}$ of the corresponding intersection numbers is even. The $q_{i} = q_i(\g)$ are sometimes called the \textit{length parameters} of $\g$.
 
To define the  \textit{twist parameter} $\tw_i = \tw_i(\g) \in \ZZ$ of $\g$ about $\s_i$, we first have to fix a marking on $\Sigma$, for example by fixing a specific choice of dual curve $D_i$ to each pants curve $\s_i$, see Section~\ref{sec:dylantwist} below. Then, after isotoping $\g$ into a well-defined standard position relative to $\P$ and to the marking, the twist $ \tw_i$ is the signed number of times that $\g$ intersects  a  short  arc  transverse to $\s_i$. We make the convention that if $i(\g,\s_i) = 0$, then $\tw_i(\g) \geqslant 0$ is the number of components in $\g$ freely homotopic to $\s_i$.

There are various ways of defining the \emph{standard position} of $\g$, leading to differing definitions of the twist. The parameter $\tw_i(\g) = p_i(\g)$ which occurs in the statement of Theorem~\ref{thm:traceformula} is the one defined by Dylan Thurston~\cite{Dylan}, however in the proof of the formula we will find it convenient to use a slightly different definition $\tw_i(\g) = \hat p_i(\g)$ given by Penner~\cite{Penner}. Both of these definitions are explained  in detail below, as is the precise relationship between them. With either definition, a classical theorem of  Dehn~\cite{Dehn}, see also~\cite{Penner} (p.12), asserts that the length and twist parameters uniquely determine $\g$:

\begin{Theorem}\textbf{(Dehn's theorem)}\\ 
The map $\Psi: \S(\Sigma) \to \Z_{\geqslant 0}^{\xi} \times \Z^{\xi}$ which sends $\g \in \S(\Sigma)$ to  \\ $(q_{1}(\g), \ldots, q_{\xi}(\g);\tw_{1}(\g), \ldots, \tw_{\xi}(\g))$ is an injection. The point \\ $(q_{1}, \ldots, q_{\xi}, \tw_{1}, \ldots, \tw_{\xi}) $  is in the image of $\Psi$ (and hence corresponds to a  curve) if and only if:
\begin{enumerate}
	\renewcommand{\labelenumi}{(\roman{enumi})}
		\item if $q_{i} = 0$, then $\tw_{i} \geqslant 0$, for each $i = 1, \ldots, \xi$.
		\item if $\s_{i_{1}}, \s_{i_{2}}, \s_{i_{3}}$ are pants curves which together bound a pair of pants whose interior is embedded in $\Sigma$, then the sum $q_{i_{1}} + q_{i_{2}} + q_{i_{3}}$ of the corresponding intersection numbers is even.
\end{enumerate} 
\end{Theorem}
		
We remark that as a special case of (ii), the intersection number  with a pants curve which bounds an embedded once-punctured torus or  twice-punctured disk in $\Sigma$ is even. 

One can think of this theorem in the following way. Suppose given a  curve $\g\in \S$, whose  length parameters $q_{i}(\g)$ necessarily satisfy the parity condition (ii). Then the $q_{i}(\g)$ uniquely determine $\g \cap P_{j}$ for each pair of pants $P_{j}$, $j= 1, \ldots, k$, in accordance with the possible arrangements of arcs in a pair of pants, see for example~\cite{Penner}. Now given two pants adjacent along the curve $\s_{i}$, we have $q_{i}(\g)$ points of intersection coming from each side and we have only to decide how to match them together to recover $\g$. The matching takes place in the cyclic cover of an annular neighbourhood of $\s_i$. The twist parameter $tw_i(\g)$ specifies which of the $\Z$ possible choices is used for the matching. 
 
\subsection{The DT-twist}
\label{sec:dylantwist}

In~\cite{Dylan}, Dylan Thurston gives a careful definition of the twist $\tw_i(\g) = p_i(\g)$ of $\g \in \S$ which is essentially the `folk' definition and the same as that implied in~\cite{flp}. He observes that this definition has a nice intrinsic characterisation, see Section~\ref{sub:intrisic} below. Furthermore, it  turns out to be the correct definition for our formula in Theorem~\ref{thm:traceformula}.

\subsubsection{The marking}
\label{sec:marking}

Given the pants decomposition $\P $ of $\Sigma$, we note, following~\cite{Dylan}, that we can fix a marking on $\Sigma$ in three equivalent ways. These are: 
\begin{enumerate}
\renewcommand{\labelenumi}{(\alph{enumi})}
    \item a \textit{reversing map}: an orientation-reversing map $R: \Sigma \to \Sigma$ so that for each $i = 1, \ldots, \xi$ we have $R(\s_{i}) = \s_{i}$;
	\item a \textit{hexagonal decomposition}: this can be defined by a   curve which meets each pants curve twice, decomposing each pair of pants into two hexagons; 
	\item \textit{dual curves}: for each $i$, a curve $D_{i}$ so that $i(D_{i}, \s_{j}) = 2\delta_{ij}$.
\end{enumerate}

The characterisations (a) and (b)  are most easily understood in connection with a particular choice of hyperbolic metric on $\Sigma$. Recall that a pair of pants  $P$ is the union  of two right angle hexagons glued along  its seams. There is an orientation reversing symmetry of $P$ which fixes the seams. The endpoints of exactly two seams meet each component of the boundary $\dd P$. Now let $\Sigma_0$ be a hyperbolic surface formed by gluing  pants $P_1, \ldots,P_k$ in such a way that the seams are exactly matched on either side of each common boundary curve $\s_i$. In this case the existence of the orientation reversing map $R$ as in (a) and the hexagonal decomposition as in (b)  are clear and are clearly equivalent. 

If the modular surface associated to $\s_i$ is made up of two distinct pants $P,P'$, then, as explained above, the dual curve $D_i$ to $\s_i$ is obtained by gluing the two $scc$-arcs  in $P$ and in $P'$ which run from $\s_i$ to itself.  Each arc meets $\s_i$ orthogonally so that in the metric $\Sigma_0$ the two endpoints on each side of $\s_i$ are exactly matched by the gluing. If the modular surface is a single pair of pants $P$, then the dual curve is obtained by gluing the single $dcc$-arc in $P$ which runs from $\s_i$ to itself. Once again both ends of this arc meet  $\s_i$ orthogonally and in the metric $\Sigma_0$ are exactly matched by the gluing. In this case,  following the convention explained in Section~\ref{sec:background}, we take the dual curve $D_i$ to be two parallel copies of the loop just described. Thus in all cases $i(D_{i}, \s_{j}) = 2\delta_{ij}$ and furthermore   the curves $D_{i}$ are  fixed by $R$.  

A general surface $\Sigma$ can be obtained from $\Sigma_0$ by performing a Fenchel-Nielsen twist about each $\s_i$.  Namely, if $A_i = \s_i \times [0,1]$ is an annulus around $\s_i$ and if we parameterise $\s_i$ as $s \mapsto \s_i(s) \in \Sigma$ for $s \in [0,1)$, then the distance $t$ twist, denoted $FN_t: \Sigma_0 \to \Sigma$, maps $A_i$ to itself by  $(\s_i(s),\theta) \mapsto (\s_i(s+\theta t),\theta)$ and is the identity elsewhere. Clearly $FN_t$ induces a reversing map, a hexagonal decomposition, and dual curves on the surface $FN_t (\Sigma)$, showing that each of (a), (b) and (c) equivalently define  a marking on an arbitrary surface $\Sigma$.

\subsubsection{The twist}
\label{sec:twist2} 

Having defined the marking, we can now define the twist $p_i(\g)$ for any $\g \in \S$. Arrange, as above, the dual curves $D_{i}$ to be fixed by $R$, so that, in particular, if $\s_i$ is the boundary of a single pair of pants $P$, then the two parallel components of the curve $D_i$ are contained one in each of the two hexagons making up $P$. For each $i = 1, \ldots, \xi$, choose a small annular neighbourhood $A_{i}$ of $\s_{i}$, in such a way that  the complement $\Sigma - \cup_{i = 1}^{\xi} \mathrm{Int}(A_i)$ of the interiors of these annuli in $\Sigma$ are pants $\hat P_{1}, \ldots, \hat P_{k}$. Arrange $\g$ so that its intersection with each  $\hat P_{i}$ is fixed by $R$ and so that  it is transverse to $D_{i}$. Also push any component of $\g$ parallel to any  $\s_{i}$ into $A_{i}$. 
  
If $q_i = i(\g,\s_i)=0$, define $p_{i} \geq 0$ to be the number of components of $\g$ parallel to $\s_i$.  Otherwise,  $q_i = i(\g,\s_i)> 0$. In this case, orient both $\g \cap A_{i}$ and $D_{i} \cap A_{i}$ to run consistently from one boundary component of $A_{i}$ to the other. (If $M(\s_i)$ is $\Sigma_{0,4}$, then the two arcs of $D_i \cap A_i$ will be oriented in opposite directions relative to the connected curve $D_i$.) Then define $$p_{i} =\hat{i}(\g \cap A_{i},D_{i} \cap A_{i}),$$ where $\hat i (\a,\beta)$ is the algebraic intersection number between the curves $\a$ and $\beta$, namely the sum of the indices of the intersection points of $\alpha$ and $\beta$, where an intersection point is of index $+1$ when the orientation of the intersection agrees with the orientation of $\Sigma$ and $-1$ otherwise.

Note that this definition is independent of both the choice of the orientations of $\g \cap A_{i}$ and $D_{i} \cap A_{i}$, and of the choice of the arrangement of $\g$ in the pants adjacent to $\s_i$. Also note  that, following the convention about dual curves in Section~\ref{sec:dualcurves},  $p_i$ is always even.  Two simple examples are illustrated in Figures~\ref{figure1_13} and~\ref{figure1_14}.

\subsubsection{An alternative definition}
\label{sec:alternative} 

The twist $p_i$  can also be described in a slightly different way as follows. Lift $A_i$ to its $\ZZ$-cover which is an infinite strip $H$.  As shown in Figures~\ref{figure1_13} and~\ref{figure1_14}, the lifts  of $D_i \cap A_i$ are arcs joining the two boundaries $\dd_0 H$ and $\dd_1H$ of $H$. They are equally spaced like rungs of a ladder in such a way that there are exactly two lifts in any period of  the translation corresponding to $\s_i$. Any arc of $\g$ enters $H$ on one side and leaves on the other.  Fix such a rung $D_*$ say and number the strands  of $\g$ meeting $\dd_0 H$  in order as $X_n, n \in \ZZ$, where $X_0$ is the first arc to the right of $D_*$  and $n$ increases moving to the right along $\dd_0 H$, relative to the orientation of the incoming strand of $\gamma$. Label the endpoints of $\g$ on $\dd_1H$ by $X'_n, n \in \ZZ$  correspondingly, as shown in Figure~\ref{figure1_13}. Since $\g$ is simple, if $X_0$ is matched to  $X'_r$,  then $X_n$  is matched to $X'_{n+r}$ for all $n \in \ZZ$.  Then it  is not hard to see that $r=p_i/2$.

\begin{figure}[hbt] 
\centering 
\includegraphics[height=3.5cm]{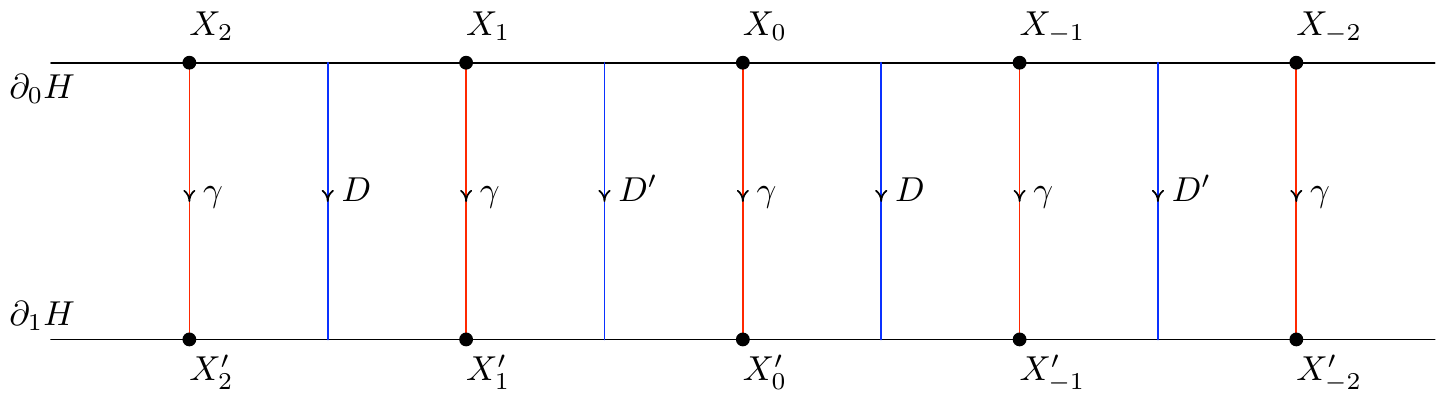} 
\caption{A curve $\g$ with $p_i(\g) = 0$. The arcs $D,D'$ together project to the dual curve $D_i$.}
\label{figure1_13}
\end{figure}

\begin{figure}[hbt] 
\centering 
\includegraphics[height=3.5cm]{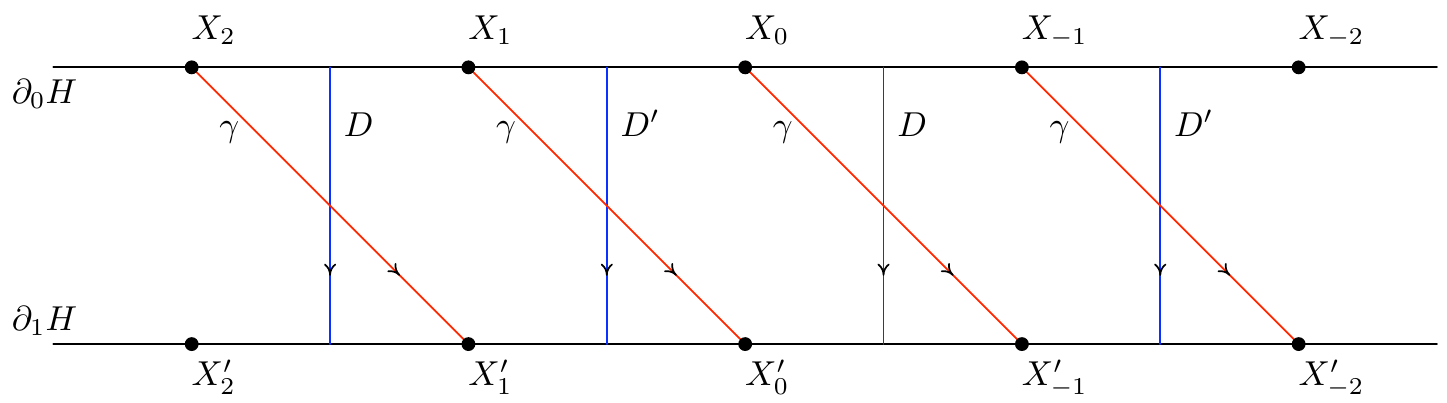} 
\caption{A curve $\g$ with $p_i(\g) = -2$.}
\label{figure1_14}
\end{figure}

\subsubsection{Intrinsic characterisation}
\label{sub:intrisic}
 
The intrinsic characterisation of the twist  in~\cite{Dylan} uses the \emph{Luo product} $\a \cdot \b $ of curves $\a,\b \in \S$  on an oriented surface $\Sigma$. This is defined as follows~\cite{Luo2, Dylan}:

\begin{itemize}
	\item If $a\cap b = \emptyset$, then $\a \cdot \b = \a \cup \b \in \S$.
	\item Otherwise, arrange $\a $ and $\b $ in minimal position, that is, such that $i(\a \cap \b) = |\a \cap \b|$. In a neighbourhood of each intersection point $x_j \in \a\cap \b$,  replace  $\a \cup \b$ by  the union of the two arcs which turn  left from $\a$ to $\b$ relative to the orientation of $\Sigma$,  see  Figure \ref{Fig:resolution}. (In~\cite{Luo2} this is called the \emph{resolution} of $\a \cup \b$ from $\a$ to $\b$ at $x_j$.) Then $\a \cdot \b$ is the curve made up from  $\a \cup \b$ away from the points $x_j$, and the replacement arcs near each $x_j$.
\end{itemize} 

\begin{figure}[hbt] 
\centering 
\includegraphics[height=4cm]{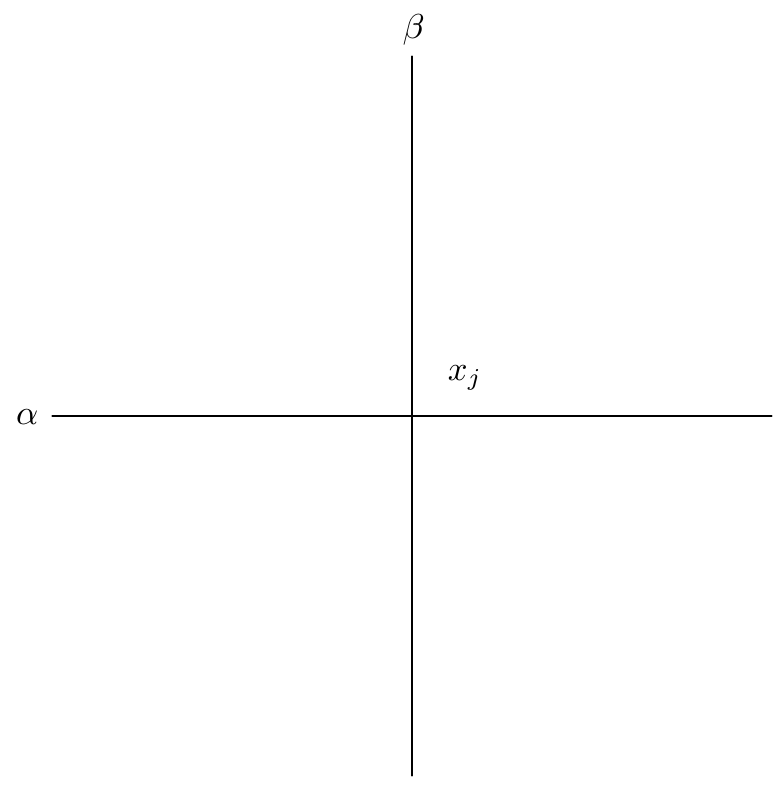}
\hspace{0.5cm} 
\includegraphics[height=4cm]{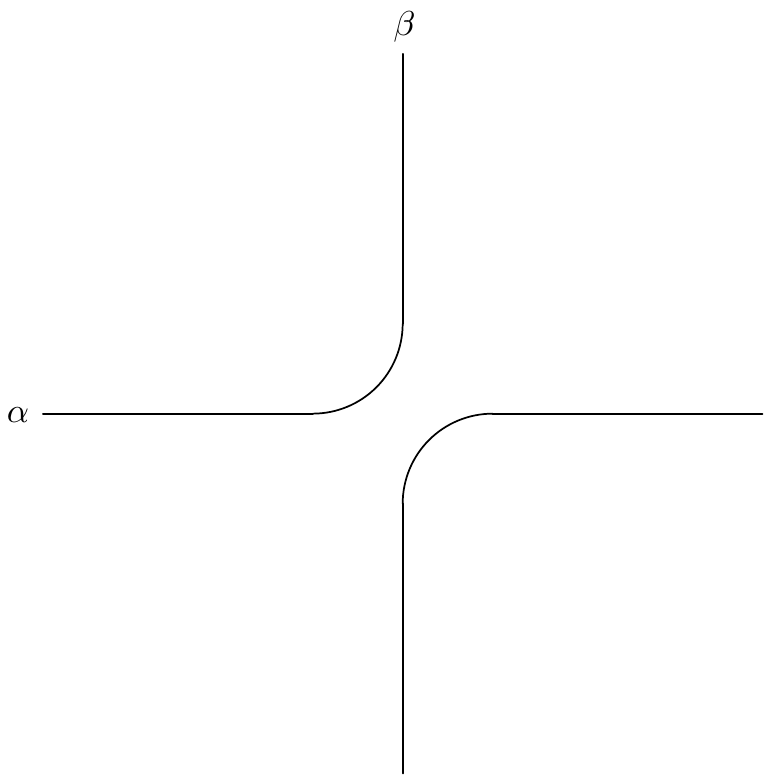} 
\caption{The Luo product: the resolution of $\a \cup \b$ at $x_j$.}
\label{Fig:resolution}
\end{figure}
  
\begin{Proposition}[\cite{Dylan} Definition 15]
\label{prop:dylantwist}
	The function $p_{i}:\S(\Sigma) \to \ZZ$ is the unique function such  that for all $\g \in \S$:
	\begin{enumerate}
	\renewcommand{\labelenumi}{(\roman{enumi})}
		\item $p_{i}(\s_{j} \cdot \g) = p_{i}(\g) + 2\delta_{ij}$;
		\item $p_{i}$ depends only on the restriction of $\g$ to the pants adjacent to $\s_{i}$; 
		\item $p_{i}(R(\g)) = -p_{i}(\g)$, where $R$ is the orientation reversing involution of $\Sigma$ defined above. 
	\end{enumerate} 
\end{Proposition}

We call $p_i(\g)$ the \emph{DT-twist parameter} of $\g$ about $\s_i$. Property (i)  fixes our convention noted above that the \emph{right} twist is taken  positive. Notice that  $p_i(D_i) = 0$. We also observe:

\begin{Proposition} Let $\g \in \S$. Then
$$p_{i}\left({\Tw}_{\s_i}(\g)\right) = p_{i}(\g)+2q_{i}.$$
\end{Proposition} 

\subsubsection{Relation to~\cite{flp}}
 
In~\cite{flp},  a curve $\g \in \S$ is parameterized by three non-negative integers $(m_i,s_i, t_i)$. These are defined as the intersection numbers of $\g$ with the three curves $K_i$, $K'_i$ and $K''_i$, namely the pants curve $\s_i$, its dual curve $D_i$, and ${\Tw}_{\s_i}(D_i)$, the right Dehn twist of $D_i$ about $\s_i$, see Figure 4 on p. 62 in~\cite{flp}. In particular:
\begin{itemize}
	\item $m_i(\g) = i(\g,K_i) = i(\g, \s_i) = q_i(\g)$
	\item $s_i(\g) = i(\g,K'_i) = i(\g, D_i) = \dfrac{|p_i(\g)|}{2}$
	\item $t_i(\g) = i(\g,K''_i) = i\left(\g, {\Tw}_{\s_i}(D_i)\right) =\bigl  |\dfrac{p_i(\g)}{2}-q_i(\g)\bigr |$
\end{itemize}
 
As proved in~\cite{flp}, the three numbers $m_i,s_i$ and $t_i$ satisfy one of the three relations $m_i = s_i +t_i$; $s_i = m_i +t_i$; $t_i = m_i +s_i$. As it is easily verified by a case-by-case analysis, we have:

\begin{Theorem}
	Each triple $(m_i,s_i, t_i)$ uniquely determines and is determined by the parameters $q_i$ and $p_i$. In fact, $q_i = m_i$ and $p_i =  2sign(p_i)s_i$
	where $$sign(p_i) = \begin{cases} +1 & \text{if\;} m_i = s_i +t_i \text{\;\;or\;} s_i = m_i +t_i;\\
	                                  -1 & \text{if\;} t_i = m_i +s_i.
		                \end{cases}$$
\end{Theorem}

\begin{proof}
If $sign(p_i) = -1$, then $p_i/2 -q_i \leq 0$. So \\
$t_i = |\dfrac{p_i}{2}-q_i| = -(\dfrac{p_i}{2}-q_i) = \dfrac{|p_i|}{2} +q_i = s_i + m_i$; 
	
If $sign(p_i) = +1$, then:	
\begin{enumerate}
		\item if $\dfrac{p_i}{2} \leqslant q_i$, then $t_i = |\dfrac{p_i}{2}-q_i| = q_i - \dfrac{|p_i|}{2} = m_i - s_i$;
		\item if $q_i \leqslant \dfrac{p_i}{2}$, then $t_i = |\dfrac{p_i}{2}-q_i| = \dfrac{|p_i|}{2} - q_i = s_i - m_i$,
\end{enumerate}
as we wanted to prove.
\end{proof}

\subsection{The P-Twist}
\label{sec:pennertwist}

We now summarise Penner's definition of the twist parameter following \cite{Penner} Section 1.2.  Instead of arranging the arcs of $\g$ transverse to $\s_i$ symmetrically with respect to the involution $R$,  we now arrange them to cross $\s_i$ through a  short closed arc $w_i \subset \s_i$. There is  some choice to be made in how we do this, which leads to the difference with the definition of the previous section. It is convenient to think of $w_i$ as contained in the two `front' hexagons of the pants $P$ and $P'$ glued along $\s_i$, which we  will also refer to as the `white' hexagons. 

Precisely, for each pants curve $\s_i \in \PC$, fix a short closed arc $w_i \subset \s_i$, which we take to be symmetrically placed in the white hexagon of one of the adjacent pants $P$, midway between the two seams   which meet $\s_i \subset \dd P$. For each $\s_i$, fix an annular neighbourhood $A_i$ and extend $w_i $ into a `rectangle'
$R_i \subset A_i$ with one edge on each component of $\partial A_i$ and `parallel' to $w_i$ and two edges arcs from one component of $\partial A_i$ to the other. (See \cite{Penner} for precise details.)
 
Now  isotope $\g \in \S$ into \emph{Penner standard position} as follows. Any component of $\g$ homotopic to $\s_i$ is isotoped into $A_i$. Next, arrange $\g$ so that it intersects each $\s_i$ exactly $q_i(\g)$ times and moreover so that all points in $\g \cap \s_i$ are contained in $w_i$. We further arrange that all the twisting of $\gamma$ occurs in $A_i$. Precisely, isotope so that $\gamma \cap \partial A_i \subset \partial R_i$, in other words, so that $\gamma$ enters $A_i$ across the edges of $R_i$ parallel to $w_i$. By pushing all the twisting into $A_i$, we can also arrange that outside $A_i$, any $dcc$-arc of $\g \cap P$ does not cross  any seam of $P$.  The $scc$-arcs are slightly more complicated. Any such arc  has both endpoints on the same boundary component, let say $\dd_0P$. Give the white hexagon (the `front' hexagon in Figure~\ref{fig:Pennerstandard}) the same orientation as the surface $\Sigma$. With this orientation, one of the two other boundary components, say $\dd_1P$, is to the right of $\dd_0 P$ and the other, say $\dd_\infty P$, to the left. We isotope the $scc$-arc so that outside $A_i$  it loops round the \emph{right hand} component $\dd_1P$, cutting the seam which is to the right of the seam contained in $\dd_0 P$ exactly, see Figure~\ref{fig:Pennerstandard}. 

\begin{figure}[hbt] 
\centering 
\includegraphics[height=7cm]{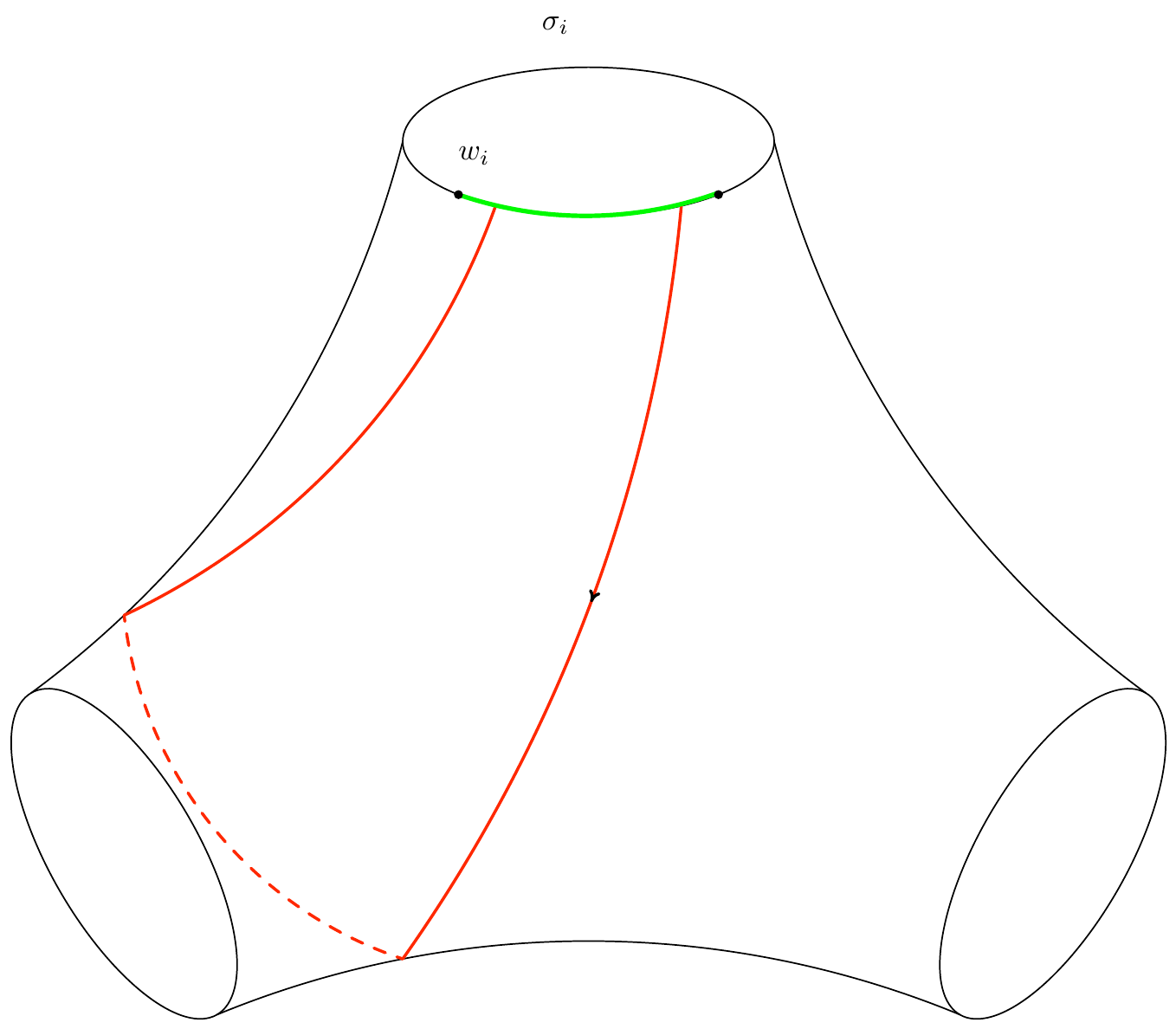} 
\caption{An scc-arc in Penner standard position.}
\label{fig:Pennerstandard}
\end{figure}

Having put $\g$ into Penner standard position, we define the \emph{Penner-twist} or \emph{$P$-twist} $\hat p_i(\g)$ as follows. Let $d_i$ be a  short arc transverse to $w_i$ with one endpoint on each of the two components of $\dd A_i$. 
\begin{itemize}
	\item If $q_{i}(\g) =i(\g,\s_i) = 0$, let $\hat p_{i}(\g) \geqslant 0$ be the number of components of $\g$ which are freely homotopic to $\s_i$. 
	\item If $q_{i}(\g) \neq 0$,  let $|\hat p_{i}(\g)|$  be the minimum number of arcs of $\g \cap A_i$ which intersect $d_i$, where the minimum is over all families of arcs properly embedded in $A_i$, isotopic to $\g\cap A_i$ by isotopies fixing $\partial A$ pointwise. Take $\hat p_{i}(\g) \geqslant 0$  if some components of $\g$ twist to the right in $A_i$ and $\hat p_{i}(\g) \leqslant 0$ otherwise. (There cannot be components twisting in both directions since $\g$ is embedded and, if there is no twisting, then $\hat p_{i}(\g) = 0$.) 
\end{itemize}

\begin{figure}[hbt] 
\centering 
\includegraphics[height=13cm]{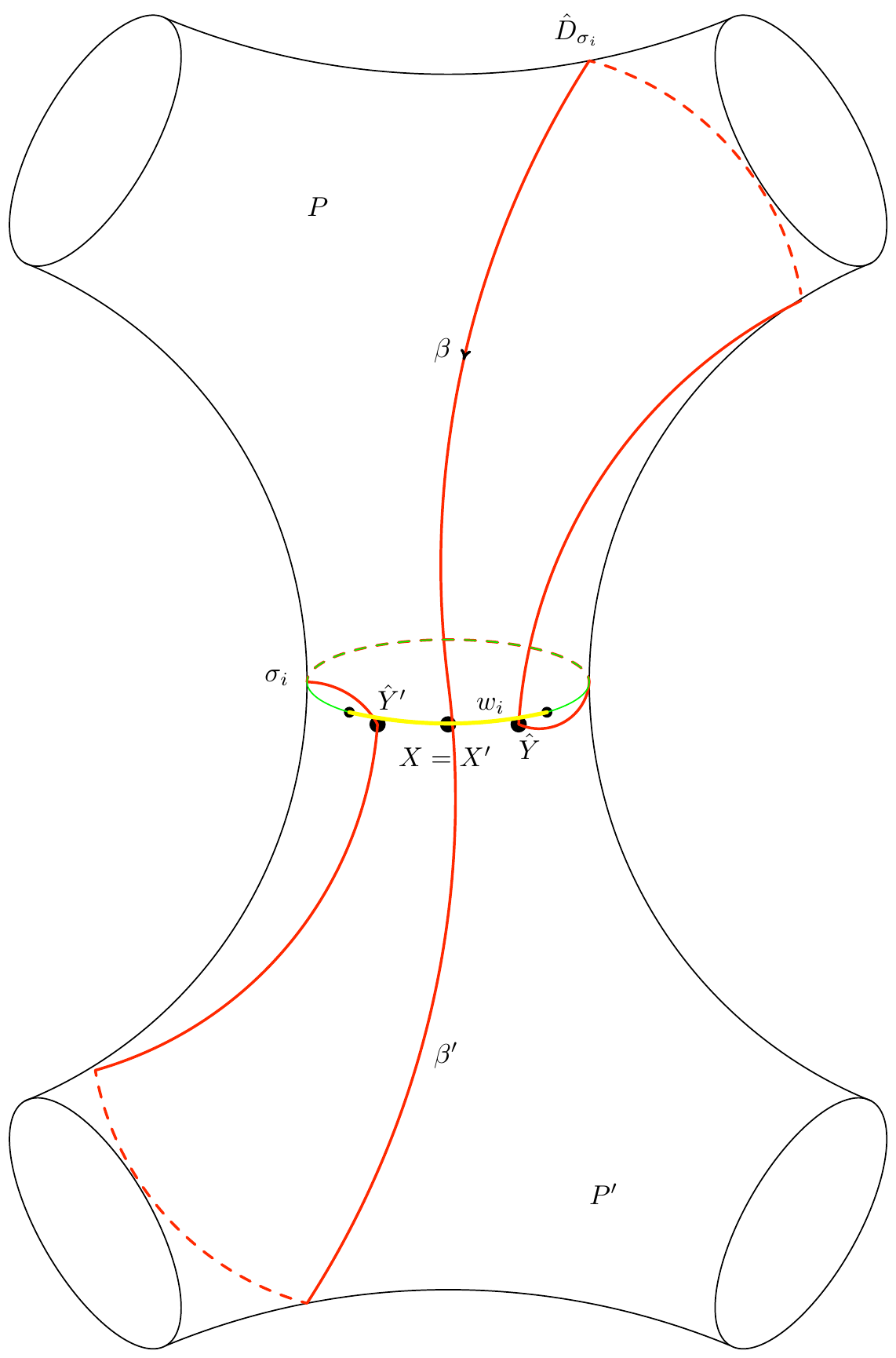} 
\caption{The dual curve $D_i$ in Penner standard position.}
\label{fig:dualPenner}
\end{figure}

\subsubsection{The dual curves in Penner position} As an example, we explain how to put the dual curves $ D_i$ into Penner standard position. This requires some care.  For clarity, we denote one component of the dual curve $D_i$ by $\hat D_i$, so that in the case in which $M(\s_i)$ is $\Sigma_{1,1}$,  we have $2\hat D_i = D_i$, while $\hat D_i = D_i$ otherwise.
 
If $M(\s_i)$ is $\Sigma_{1,1}$, there is only one arc to be glued whose endpoints we can arrange to be in $w_i$. We simply take two parallel copies of this loop $\hat D_i$ so that  $D_i = 2\hat D_i$ and $\hat p_i(D_i) = 0$.

If $M(\s_i)$ is $\Sigma_{0,4}$ then $D_i = \hat D_i$. In this case we have to match the endpoints of two $scc$-arcs  $ \beta \subset P$ and $\beta' \subset P'$, both of which have endpoints on $\s_i$. The arc $\b$ has one endpoint $X$ in the front white hexagon of $P$, which we can arrange to be in $w_i$, and the other $Y$  in the symmetrical position in the black hexagon. Label the endpoints of $\b' $ in a similar way. To get $\b \cup \b'$ into standard Penner position, we have to move the  back endpoints $Y$ and $Y'$ round to the front so that they also lie in $w_i$. Arrange $P$ and $P'$ as shown in Figure~\ref{fig:dualPenner} with the white hexagons to the front. In Penner position,  $\b$ has to loop round the right hand boundary component of $P$ so that  $Y$ has to move to a point $\hat Y$  to the right of   $X$ along $w_i$ in Figure~\ref{fig:dualPenner}. In  $P'$ on the other hand, $\b'$ has to loop round the right hand boundary component of $P'$, so that $Y'$ gets moved to a point $\hat Y'$ to the left of $X'$ on $w_i$.  Since  $X$ is identified to $X'$,  to avoid self-intersections, $\hat Y$ has  to be  joined to $\hat Y'$ by a curve which follows $\s_i$ around the back of $P \cup P'$. By inspection, we see that $\hat p_i(\hat D_i) = -1$.

\subsection{Relationship between the different definitions of twist}
\label{sec:relationtwist}

Our proof of Theorem~\ref{thm:traceformula} in Section~\ref{sec:mainthm} uses the explicit relationship between the above two definitions of the twist. The formula in Theorem~\ref{thm:twistrelation} below appears without proof in \cite{Dylan};  for completeness we supply a proof. 
 
Suppose that two pairs of pants meet along   $\s_i \in \PC$. Label  their respective boundary curves $(A,B,E)$ and $(C,D,E)$ in clockwise order, where $E = \s_{i}$, see Figure~\ref{Figure1_7}. (Some of these boundary curves may be identified in $\Sigma$.)   

\begin{Theorem} (\cite{Dylan} Appendix B)
\label{thm:twistrelation}
As above,  let $\g \in \S$ and let $q_i = q_i(\g)$, $\hat p_{i}=\hat p_{i}(\g) $ and $p_{i}= p_{i}(\g)$ respectively denote its length parameter, its P-twist and its DT-twist around $\s_i$. Then $$\hat p_{i} = \frac{p_{i}+l(A,E;B)+l(C,E;D)-q_{i}}{2},$$ where $l(X,Y;Z)$ denotes the number of strands of $\g \cap P$ running from the boundary curve $X$ to the boundary curve $Y$ in the pair of pants $P$ having   boundary curves $(X,Y,Z)$.
\end{Theorem}

\begin{proof}
Let $\gamma \in \S$. We will use  a case-by-case analysis to give a proof by induction on $n = q_{i}(\g)$.  We shall assume that the modular surface $M(\s_i)$ is $\Sigma_{0,4}$, so that  $\s_{i}$ belongs to the boundary of two distinct pants $P = (A,B,E)$ and $P'=(C,D,E)$, and  leave the case in which  $M(\s_i)$ is $\Sigma_{1,1}$ to the reader. We begin with the cases $n= 1$ and $n = 2$, because $n = 2$ is useful for the inductive step.

\begin{figure}
\centering
\includegraphics[height=7cm]{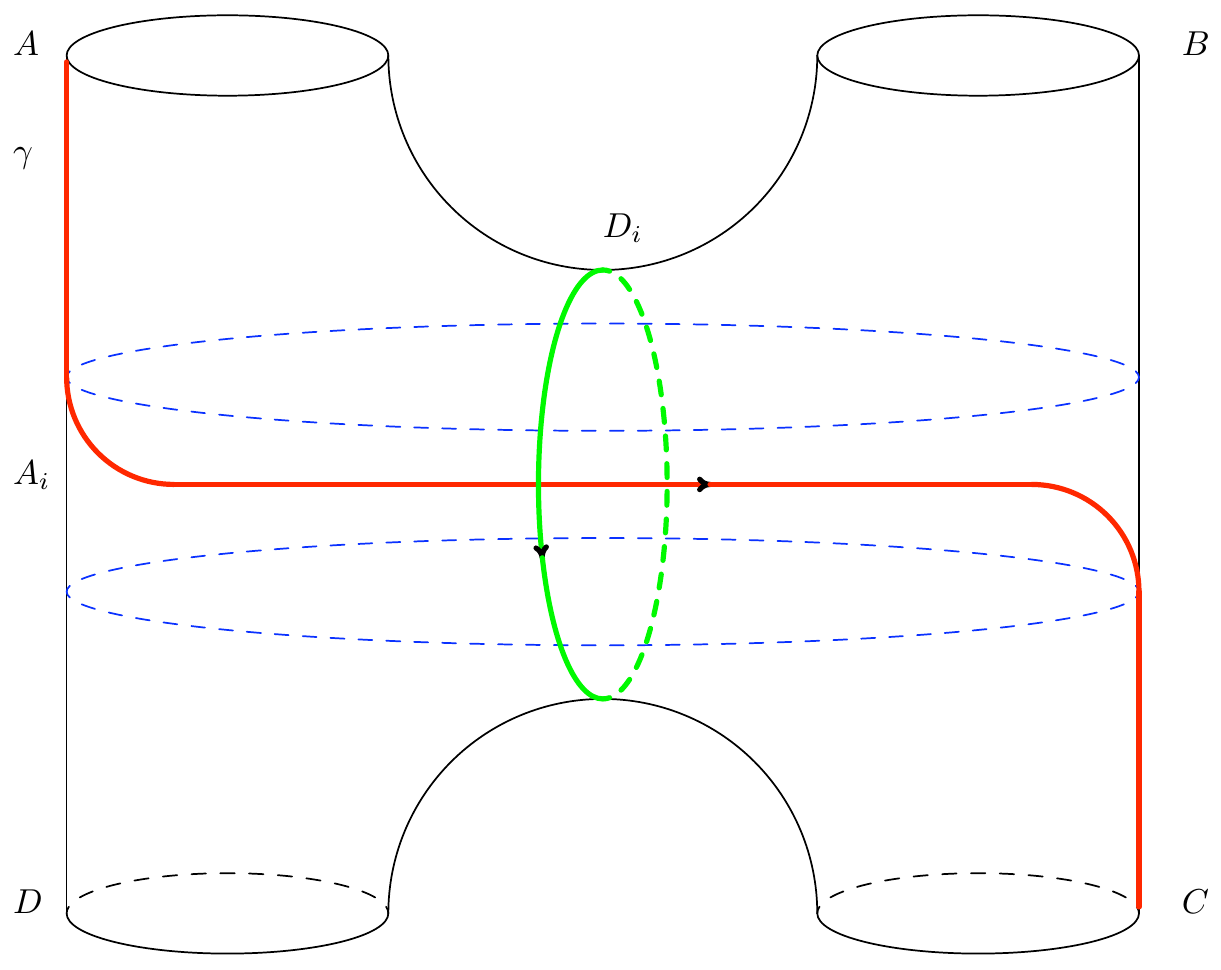}
\caption{Case $n= 1$ when $\gamma$ goes from $A$ to $C$.}
\label{Figure1_7}
\end{figure}

\medskip

When $n = 1$, the strand of $\gamma$ which intersects $\s_{i}$ must join one of the boundary components of $P$ different from $E$, to one of the two boundary components of $P'$ different from $E$. We have four cases corresponding to $\g$ joining $A$ or $B$ to $C$ or $D$. Figure \ref{Figure1_7} shows the case in which $\g$ joins $A$ to $C$. One checks easily that  $l(A,E;B) =1$, $l(C,E;D) =1$ while $\hat p = 0, p=-1$ and $q=1$, verifying the formula in this case. The other cases are similar.

\medskip

Now consider   $n = 2$, so that $\g \cap  M(\s_i)$ may have either one or two connected components. If there are two components, then each one was already analysed  in the case $k=1$, and the result follows by the additivity of the quantities involved. 

If $\g \cap  P$ is connected, we must have (in one of the pants $P$ or $P'$) an $scc$-arc  which has both its endpoints on $\s_{i}$. Without loss of generality we may suppose that this arc is in $P$. Choose an orientation on $\gamma$ and call its initial and final points $X_{1}$ and $X_{2}$ respectively. The endpoints of this arc must be joined to the boundary components $C$ or $D$ of $P'$. Figure \ref{Figure1_8} illustrates the case in which $X_{1}$ is joined to $D$, while $X_{2}$ is joined to $C$. Then $l(A,E;B) =0$, $l(C,E;D) =1$ while $\hat p = 0, p=1$ and $q=2$, verifying the formula in this case. The other cases are similar.

\begin{figure}
\centering
\includegraphics[height=7cm]{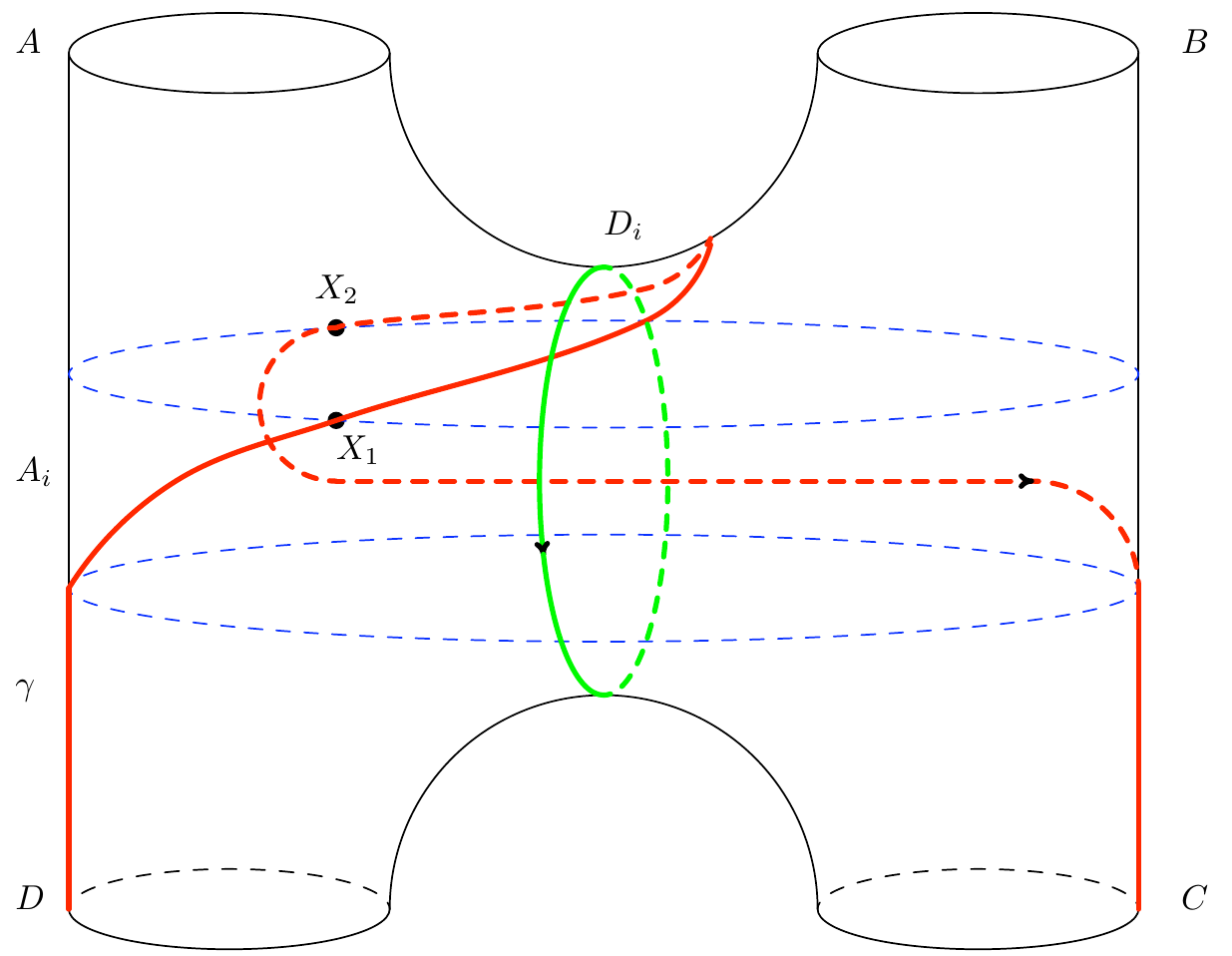}
\caption{Case $n = 2$ when $X_{1}$ is joined to $D$ and $X_{2}$ is joined to $C$. }
\label{Figure1_8}
\end{figure}

\medskip
 
Suppose now that the statement is true for any $n < q_i$. If  $\g \cap  M(\s_i)$ is not connected, then each connected component intersects  $\s_{i}$ less then $n$ times  and the result follows from the inductive hypothesis and the additivity of the quantities involved. 

If $\hat \gamma= \g \cap   M(\s_i)$ is connected, then there will be an arc which has both its endpoints on  $\s_{i}$. Choose an orientation on $\gamma$. Without loss of generality, we can suppose that the first such arc is contained in $P$. Let  $X_{1}$ and $X_{2}$ be its two ordered endpoints. Then $X_{2}$ splits $\hat \gamma$ into two oriented curves $\alpha$ and $\beta$, where $\alpha$ contains only one arc with both  endpoints in $\s_{i}$, while $\beta$ has $n-1$ arcs of this kind. Now we modify $\alpha$ and $\beta$, in such a way that they will became properly embedded arcs in $ M(\s_i)$, that is, arcs with endpoints on  $\dd ( M(\s_i)) \subset \Sigma$. We do this by adding a segment for each one of $\alpha$ and $\beta$ from $X_{2}$ to one of the boundary components $C$ or $D$ of $P'$. In order to respect the orientation of $\alpha$ and $\beta$ we add the segment twice, once with each  orientation. This will not change the quantities involved. For example, suppose we add two segments from $X_{2}$ to $C$. This creates two oriented curves $\alpha'$ and $\beta'$ in $ M(\s_i)$ such that $$t_{i}(\gamma) = t_{i}(\alpha \cup \beta) = t_{i}(\alpha' \cup \beta') = t_{i}(\alpha') + t_{i}(\beta')$$ and the conclusion now follows from the inductive hypothesis.
\end{proof}

\begin{Remark} {\rm There is a nice formula for the number $l(X,Y,Z)$ in the above theorem. Given $a, b \in \R$, let $\max {(a,b)} = a \vee b$ and $\min{(a,b)} = a \wedge b$. Suppose that a pair of pants has boundary curves $X,Y,Z$ and that $\g \in \SC$. Let $x = i(\g,X)$ and define $y,z$ similarly. As above let  $l(X,Y;Z)$  denote the number of strands of $\g$ running from $X$ to  $Y$. Then (see~\cite{Dylan} p. 20) $$l(X,Y;Z) = 0 \vee \left(\frac{x+y-z}{2} \wedge x \wedge y \right).$$}
\end{Remark}

\section{The gluing construction}
\label{sec:gluing}

As explained in the introduction, the representations which we shall consider are holonomy representations of projective structures on $\Sigma$, chosen so that the holonomies of all the loops $\s_j \in \PC$ determining the pants decomposition $\P$ are parabolic. The interior of the set of free, discrete, and geometrically finite representations of this form is  called the \emph{Maskit embedding} of $\Sigma$, see Section~\ref{sec:maskit} below.

The construction of the projective structure on $\Sigma$ is based on Kra's plumbing construction~\cite{kra}, see  Section~\ref{sec:plumbing}. However it will be convenient to describe it in a somewhat different way. The idea is to manufacture $\Sigma$ by gluing triply punctured spheres across punctures. There is one triply punctured sphere for each pair of pants $P \in \P$, and  the gluing across the pants curve $\s_j$ is implemented by a specific projective map depending on a parameter $\tau_j \in \CC$. The  $\tau_j$ will be the parameters of the resulting holonomy representation $\rho: \pi_1(\Sigma) \to PSL(2,\C)$.  

More precisely, we first fix an identification of the interior  of each pair of pants $P_i$  to a standard triply punctured sphere $\mathbb P$. We endow $\mathbb P$ with the projective structure coming from the unique hyperbolic metric on a triply punctured sphere. The gluing is carried out by deleting open punctured disk neighbourhoods of the two punctures in question and gluing horocyclic annular collars round the resulting two boundary curves, see Figure~\ref{figure2_4}. 

\begin{figure}[hbt] 
\centering 
\includegraphics[height=2.9cm]{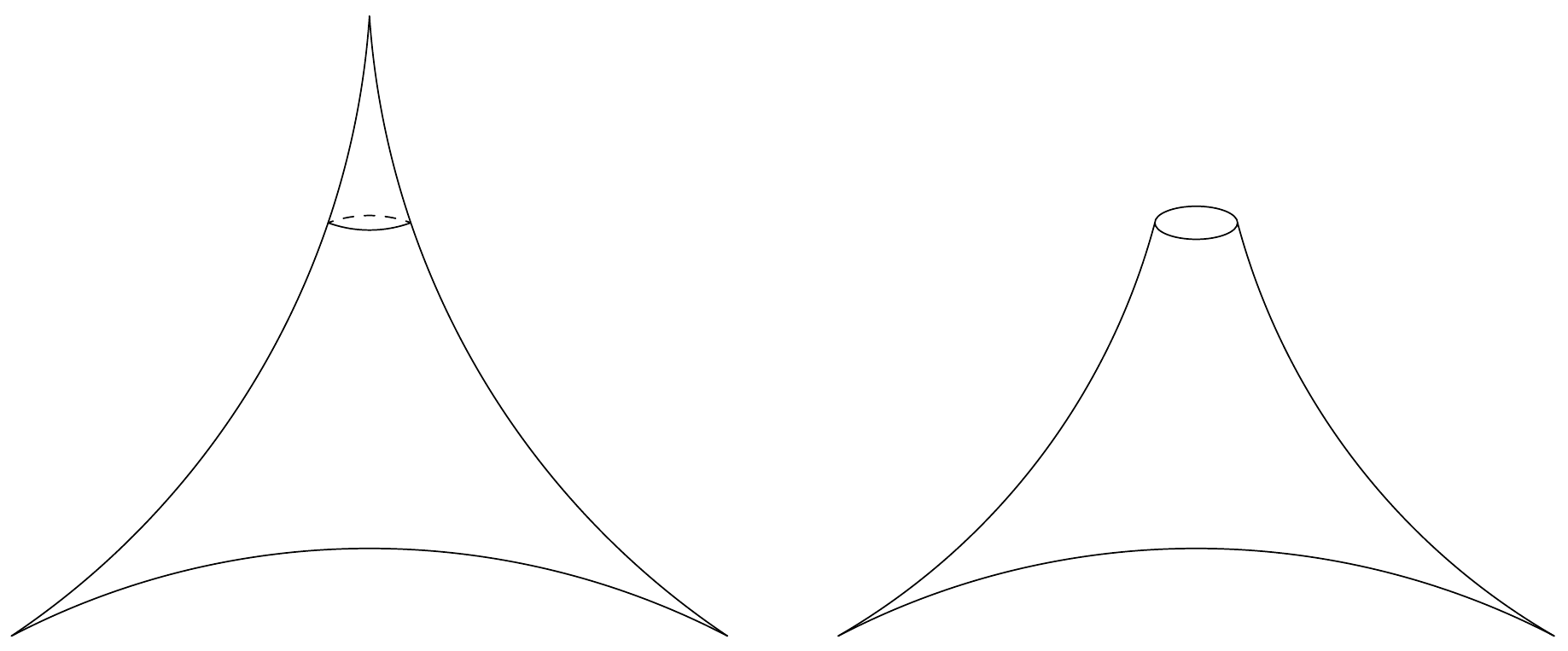} \hspace{0.4cm} \includegraphics[height=2.9cm]{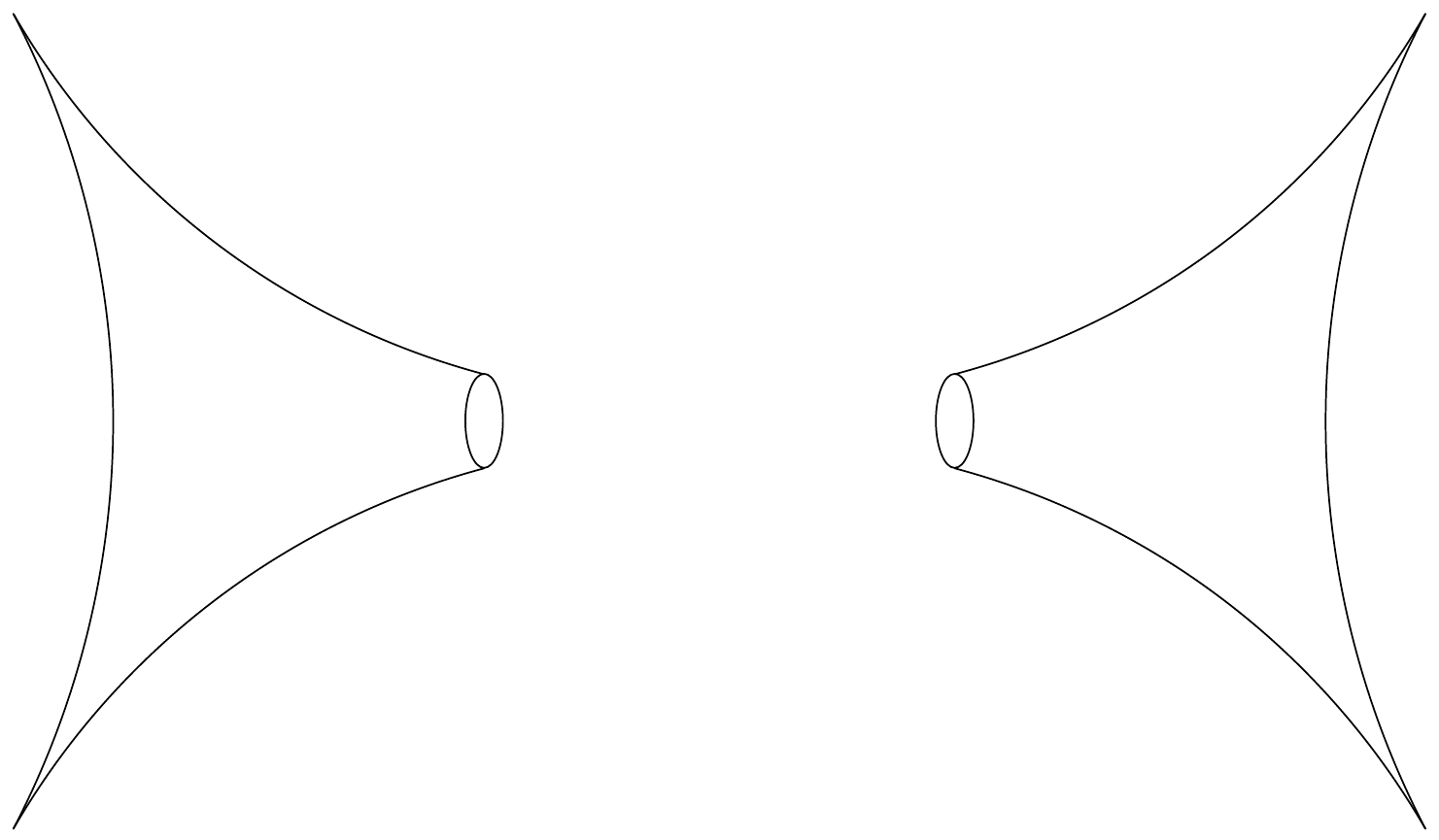} 
\caption{Deleting horocyclic neighbourhoods of the punctures and preparing to glue.}
\label{figure2_4}
\end{figure}

\subsection{The\;gluing}
\label{sec:moregluing}

To describe the gluing  in detail, first recall (see eg~\cite{Indra} p.\,207) that any triply punctured sphere is isometric to the standard triply punctured sphere $\mathbb P = \HH /\Gamma$, where $$\Gamma = 
\Bigl <\begin{pmatrix}
	1  &  2  \\
	0 &  1 \\ 
	\end{pmatrix},
	 	\begin{pmatrix}
	1  &  0  \\
	2 &  1 \\ 
	\end{pmatrix} \Bigr >.$$
Fix a standard fundamental domain for $\Gamma$ as shown in Figure~\ref{figure4_2}, so that the three punctures of $\mathbb P $ are naturally labelled  $0,1,\infty$. Let $\Delta_0$ be the ideal triangle with vertices $ \{0,1,\infty\}$, and $\Delta_1$ be its reflection in the imaginary axis. We sometimes refer to $\Delta_0$ as the white triangle and $\Delta_1$ as the black.

\begin{figure}[hbt] 
\centering 
\includegraphics[height=5cm]{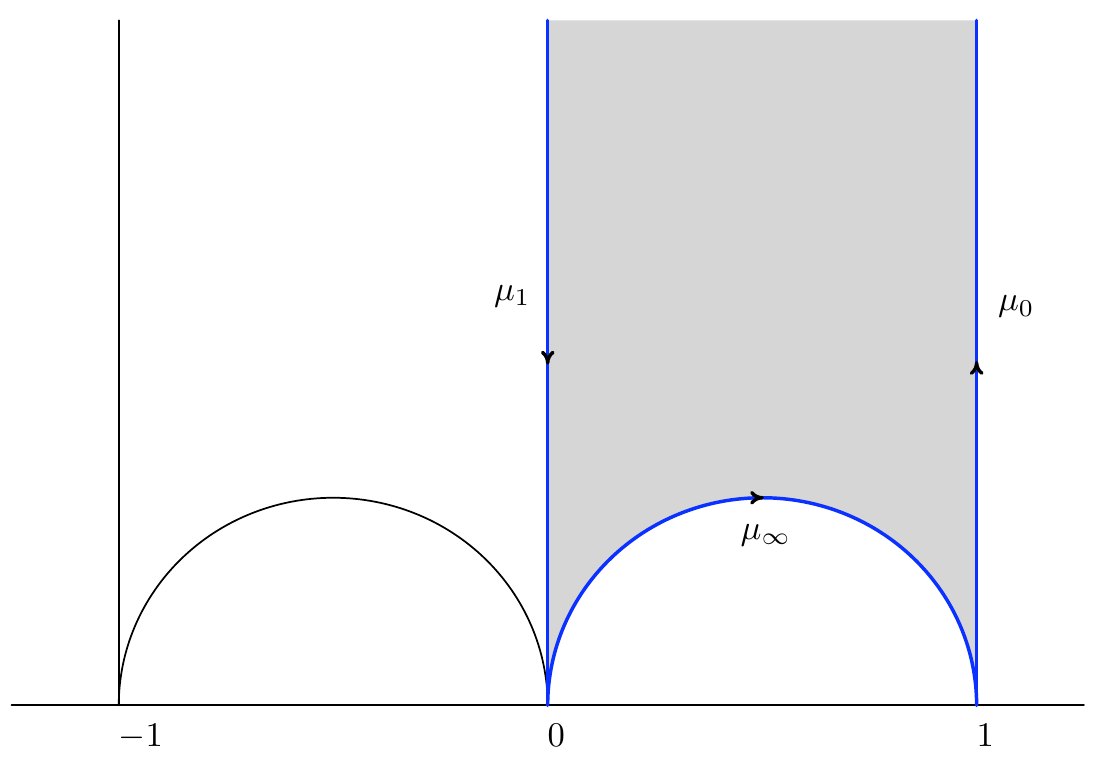} 
\caption{The standard fundamental domain  for $\mathbb P$. The shaded triangle is the white triangle $\Delta_0$.}
\label{figure4_2}
\end{figure}

\begin{figure}[hbt] 
\centering 
\includegraphics[height=15cm]{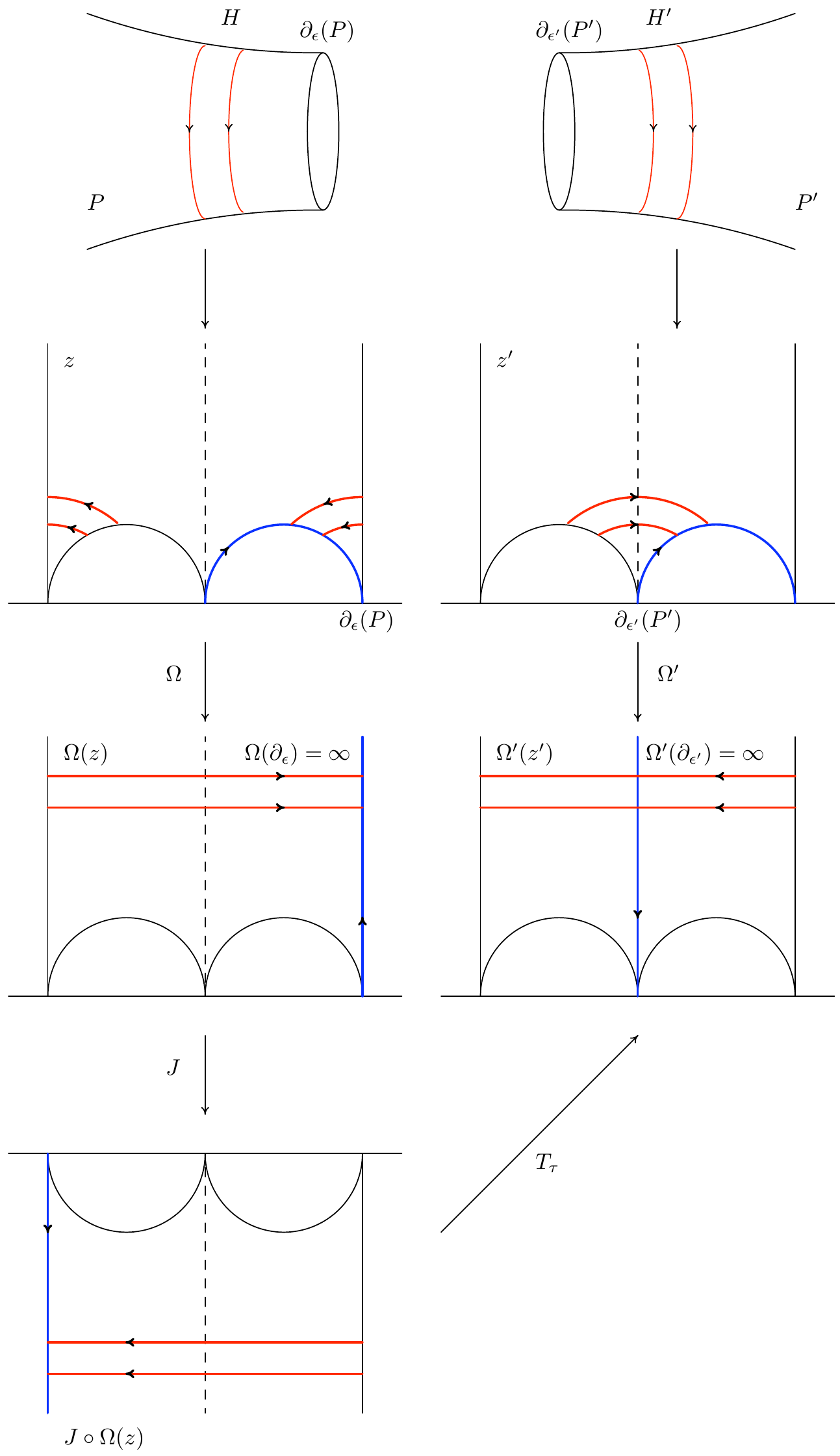} 
\caption{The gluing construction in the case $\e = 1$ and $\e' = 0$.}
\label{figure4_10}
\end{figure}

\begin{figure}[hbt] 
\centering 
\includegraphics[height=7cm]{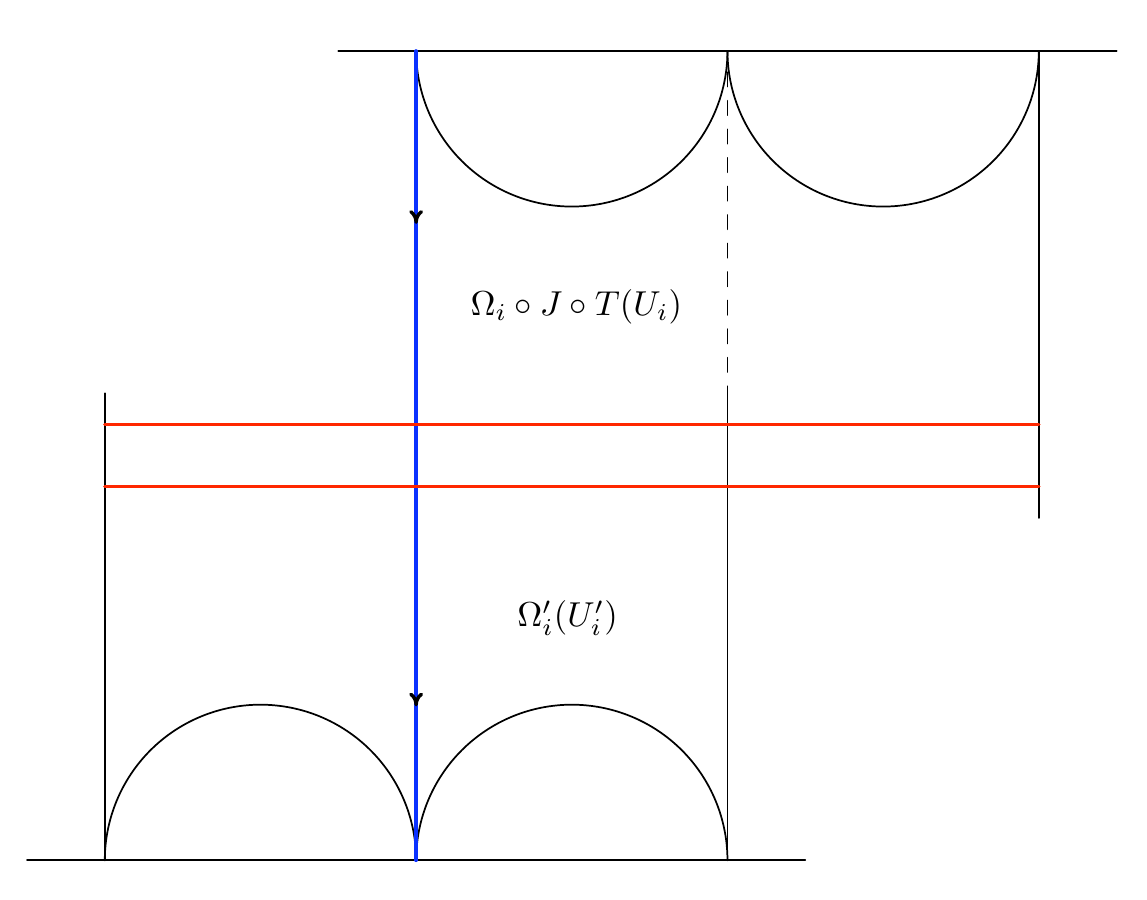} 
\caption{The gluing along horocycles where $U_i$, $U'_i$ are images of $P$, $P'$ with deleted horoballs.}
\label{figure4_11}
\end{figure}

With our usual pants decomposition  $\P$, fix homeomorphisms $\Phi_i$ from the interior of each pair of pants $P_i$ to $\mathbb P$. This identification induces a labelling of the three boundary components of $ P_{i}$ as $0, 1, \infty$ in some order, fixed from now on. We denote the boundary labelled $\e \in \{0,1,\infty \}$ by $\dd_{\e} P_{i}$. The identification also induces a colouring of the two right angled hexagons whose union is  $P_i$, one being white and one being black. Suppose that  the pants $P, P' \in \P$ are adjacent along the pants curve $\s$ meeting along boundaries  $\dd_{\e}P$ and $\dd_{\e'}P'$.   (If $P =P'$ then clearly $\e \neq \e'$.) The gluing across $\s$ will be described by a complex parameter $\tau$ with $\Im \tau >0$, called the \emph{plumbing parameter} of the gluing. We first describe the gluing in the case $\e = \e' = \infty$.
   
\medskip 
   
Arrange the pants with $P$ on the left as shown in Figure~\ref{figure4_10}. Take two copies $\mathbb P, \mathbb P'$ of $\mathbb P$. Each of these is identified with $\HH/\Gamma$ as described above. We refer to the copy of $\HH$ associated to $\mathbb P'$ as $\HH'$ and denote the natural  parameters in $\HH, \HH'$ by  $z, z'$ respectively.  Let $\zeta$  and $ \zeta'$  be the projections $\zeta: \HH \to \mathbb P$  and  $\zeta': \HH' \to \mathbb P'$ respectively.

Let $h= h(\tau) $ be the loop  on $\mathbb P $ which lifts  to the horocycle  $ \Im z= \frac{\Im \tau}{2}$ on $\HH$.  For  $\eta>0$, $H = H(\tau,\eta) = \frac{\Im \tau-\eta}{2} \leqslant \Im z \leqslant  \frac{\Im \tau+\eta}{2} \subset \HH$ is a horizontal strip which projects to an annular neighbourhood $A$ of $h \subset \mathbb P$. Let $S \subset \mathbb P$ be the surface $\mathbb P$ with the  projection of the  open horocyclic neighbourhood $\Im z >  \frac{\Im \tau+\eta}{2}$ of $\infty$ deleted. Define $h', S'$ and $A'$ in a similar way.  We are going to glue $S$ to $S'$ by matching $A $ to $A'$  in such a way that $h$ is identified to $h'$ with orientation reversed, see Figures~\ref{figure4_10} and \ref{figure4_11}. The resulting homotopy class of the loop $h$ on the glued up surface (the quotient of the disjoint union of the surfaces $S_i$ by the attaching maps across the $A_i$) will be in the homotopy class of $\s$.    To keep track of the marking on $\Sigma$, we will do the gluing  on the level of the $\ZZ$-covers of $S, S'$ corresponding to $h,h'$, that is, we will actually glue the strips $H$ and $H'$. 

As shown in Figure~\ref{figure4_10}, the deleted punctured disks are on opposite sides of $h$  in $S$ and $h'$ in $S'$. Thus we first need to reverse the direction in one of the two strips $H, H'$. Set
\begin{equation}
\label{eqn:standardsymmetries1}
J = \begin{pmatrix}
	-i  &  0  \\
	0 &  i   
	\end{pmatrix},  \ \ 
T_{\tau}=  \begin{pmatrix}
1 & \tau  \\
	0 &  1 \end{pmatrix}.
\end{equation}
We  reverse the direction in $H$ by applying the map $J(z) = -z$ to $\HH$.  We then glue $H$ to $H'$ by  identifying $z \in H$ to $z' = T_{\tau} J(z) \in H'$. Since both $J$ and $T_{\tau}$ commute with the holonomies $z \mapsto z+2$ and $z' \mapsto z'+2$ of the curves $h,h'$, this identification descends to a well defined identification of  $A$ with $A'$, in which the `outer' boundary $\zeta(\Im z) = (\Im \tau+\eta)/2$ of $A$ is identified to the `inner' boundary $\zeta'(\Im z') =(\Im \tau-\eta)/2$ of $A'$. In particular, $h$ is glued to $h'$ reversing orientation. 

\medskip 
    
Now we treat the general case in which $P$ and $P'$ meet along punctures with arbitrary labels  $\e, \e' \in \{0,1,\infty\}$. As above, let $\Delta_0 \subset \HH$ be the ideal `white' triangle with vertices $0,1,\infty$. Notice that there is a unique orientation preserving symmetry $\Omega_{\a}$  of $\Delta_{0}$ which sends  the vertex $\a \in \{0,1,\infty\}$ to $ \infty$: 
\begin{equation}
\label{eqn:standardsymmetries}
\Omega_{0} = \begin{pmatrix}
	1  &  -1  \\
	1 &  0   
	\end{pmatrix},  \ \ 
 \Omega_{1} =  \begin{pmatrix}
0 &  -1  \\
	1 &  -1 \end{pmatrix},  \ \ 
	\Omega_{\infty} = \Id = \ \begin{pmatrix}
	1  &  0  \\
	0 &  1 	\end{pmatrix}.
\end{equation}

To do the gluing,  first move $\e, \e'$ to $\infty$ using the maps $\Omega_{\e}, \Omega_{\e'}$ and then proceed as before. Thus the gluing identifies    $ z  \in H$ to $ z' \in H'$ by the formula 
\begin{equation}
\label{eqn:gluing} \Omega_{\e'}(z')= T_{\tau} \circ J \left(\Omega_{\e}(z)\right),
\end{equation}
see Figure~\ref{figure4_11}.

\medskip 
  
Finally, we carry out the above construction  for each pants curve $\s_i \in \PC$. To do this, we  need to  ensure that the annuli corresponding to the three different punctures  of a given $\mathbb P_i$ are disjoint. Note that the condition $\Im \tau_i > 2 $, for all $i=1,\ldots,\xi$, ensures that the three curves $h_0$, $h_1$ and $h_\infty$ associated to the three punctures of $P_i$ are disjoint in $\mathbb P$. Under this condition, we can clearly choose $\eta>0$ so that their annular neighbourhoods are disjoint, as required. 

\begin{Remark}\label{Rem:black-white}{\rm 
In the above construction of $\Sigma_0$, we  glued the two white triangles $\Delta_{0}(P)$ and $\Delta_{0}(P')$. Suppose that we wanted instead to glue the two black triangles $\Delta_{1}(P)$ and $\Delta_{1}(P')$.  This can be achieved by replacing the parameter $\tau$ with $\tau -2$. However, following our recipe, it is not possible to glue a white triangle to a black one, because the black triangle is to the right of both the outgoing and incoming lines while the white triangle is to the left. }	
\end{Remark}
   
\subsubsection{Independence of the direction of the travel}

The recipe for gluing  two pants apparently depends on the direction of travel across their common boundary. The following lemma  shows that, in fact, the gluing in either direction is implemented by the same recipe and uses the same parameter $\tau$.

\begin{Lemma}
\label{lemma:independence}
Let pants $P$ and $P'$ be glued across a common boundary $\s$, and suppose the gluing used when travelling from $P$ to $P'$ is implemented by~\eqref{eqn:gluing} with the  parameter $\tau $. Then the gluing when travelling in the opposite direction from $P'$ to $P$ is also implemented by~\eqref{eqn:gluing} with the same parameter $\tau$.
\end{Lemma}

\begin{proof}
Using the maps $\Omega_{\e}$ if necessary, we may, without loss of generality, suppose we are gluing the boundary  $\dd_{\infty}P  $ to $\dd_{\infty}P'$. (Note that $\Omega_{\infty} = \Id$.) By definition, to do this we identify the horocyclic strip $H  \subset \HH $ to the strip $H'\subset \HH'$ using the map $T_{\tau} \circ J$. 

Fix a point $X \in h $. The gluing sends $X$ to $T_{\tau} J(X) \in h'$. The gluing in the other direction from $P'$ to $P$  reverses orientation of the strips to be glued and is done using a translation $T_{\tau'}$ say. To give the same gluing we must have $T_{\tau'}JT_{\tau}J(X) = X $. This gives $ \tau' -(-X+\tau) = X$ which reduces to $\tau= \tau'$ as claimed.
\end{proof}

\subsection{Marking and Dehn twists}
\label{sec:marking2}

Write $ \ttau = ( \tau_1,\ldots,\tau_{\xi}) \in \C^{\xi} $, and denote by  $\Sigma(\ttau )$ the surface obtained by the gluing procedure described above with parameter $\tau_i$ along curve $\s_i$. To complete the description of the projective structure, we have to specify a marking on  $ \Sigma(\ttau )$. To do this we have to specify a base structure on a fixed surface $\Sigma_0$, together  with a homeomorphism $\Phi_{\ttau}: \Sigma_0 \to \Sigma (\ttau)$. 

We first fix the base structure on $\Sigma_0$, together with a marking given by a family of dual curves $D_i$ to the pants curves $\s_i$. Let  $\mu_{\e} \subset \HH$  be the unique oriented geodesic from $\e +1$ to $\e +2 $, where $\e$ is in the cyclically ordered set $\{0,1,\infty\}$, see Figure~\ref{figure4_2}.  The lines $\mu_{\e}$ project to the seams of $\mathbb P$. We call $\mu_0 $ (from $1$ to $\infty$) and $\mu_1$  (from $\infty$ to $0$) respectively the \emph{incoming} and the \emph{outgoing} strands (coming into and going out from the puncture) at $\infty$, and refer to their images under the maps $\Omega_{\e}$ in a similar way. For $\tau \in \C$, let $X_{\infty} (\tau) = 1+ \Im \tau/2$  be the point at which the incoming line $\mu_0 $ meets the horizontal horocycle $\Im z =  \Im \tau/2$ in $\HH$, and let $Y_{\infty} (\tau) = \Im \tau/2$ be the point the outgoing line $\mu_1$ meets  the same horocycle.  Also define  $X_{\e} (\tau)=\Omega_{\e}(X_{\infty})$ and $Y_{\e} (\tau)=\Omega_{\e}(Y_{\infty})$. Now pick a pants curve $\s$ and, as usual, let  $P,P' \in \P$ be its adjacent pants  in $\Sigma$, to be glued across boundaries $\dd_{\e}P$ and $\dd_{e'}P'$.  Let $X_{\e}(P), X_{\e}(P')$ be the points corresponding to $X_{\e}(\tau), X_{\e}(\tau)$ under the identifications $\Phi,\Phi'$ of $P,P'$ with $\mathbb P$, and similarly for $Y_{\e}(P), Y_{\e}(P')$. The base structure $\Sigma_0$ will be one in which the identification~\eqref{eqn:gluing} matches the point $X_{\e}(P,\tau)$ on the incoming line across $\dd_{\e}P$ to the point $Y'_{\e'}(P')$ on the outgoing line to $\dd_{\e'} P'$. Referring to the gluing equation~\eqref{eqn:gluing}, we see that this condition is fulfilled precisely when $\Re \tau = 1$.

We define the structure on $\Sigma_0$ by specifying $\Re \tau_i = 1$ for $i = 1,\ldots,\xi$. The imaginary part of $\tau_i$ is unimportant; for definiteness we can fix $\Im \tau_i  = 4$. Now note that the reflection $z \mapsto -\bar z$ of $\HH$ induces an orientation reversing isometry of $\mathbb P$ which fixes its seams; with the gluing matching seams as above this extends, in an obvious way, to an orientation reversing involution of $\Sigma_0$. Following (a) of Section~\ref{sec:marking}, this specification is equivalent to a specification of a marking on $\Sigma_0$.  

Finally, we define a marking on the  surface $\Sigma(\ttau)$. After applying  a suitable stretching to each pants to adjust the lengths of the boundary curves, we can map $  \Sigma_0 \to \Sigma(\ttau)$  using a map which is the Fenchel-Nielsen twist $FN_{\Re \tau_i}$  on an annulus around   $\s_i \in \PC, i=1,\ldots, \xi$ and the identity elsewhere, see Section~\ref{sec:marking}. This gives a well defined homotopy class of homeomorphisms $\Psi_{\ttau}: \Sigma_0 \to \Sigma(\ttau)$.
 
With this description, it is easy to see that $\Re \tau_i$ corresponds to twisting about $\s_i$; in particular,  $\tau_i \mapsto \tau_i + 2$ is a full right Dehn twist about $\s_i$. The imaginary part $\Im \tau_i$ corresponds to vertical translation and has the effect of scaling the lengths of the $\s_i$.

\subsection{Projective structure and holonomy}
\label{sec:projstructure}

The above construction gives a way to define the developing map from the universal cover $\tilde \Sigma$ of $\Sigma$ to $\CC$.  To do this we have to describe the developing image of any path $\g$ on $\Sigma$. The path goes through a sequence of pants $P_{i_1}, P_{i_1},\ldots, P_{i_n}$  such that each adjacent pair $P_{i_j},P_{i_{j+1}}$ is glued along an annular neighbourhood $A(\s_{i_j})$ of the pants curve $\s_{i_j}$ which forms the common boundary of $P_{i_j}$ and $P_{i_{j+1}}$. Since all the maps involved in this gluing are in $PSL(2,\C)$, it is clear that if $\g$ is a closed loop, then the holonomy of $\g$ is in $PSL(2,\C)$. Thus we get a representation $\pi_1(\Sigma) \to PSL(2,\C)$ which can be checked to be independent of $\g$ up to homotopy in the usual way.

Now we can justify our claim that our construction gives a projective structure on $\Sigma$. Recall that a complex projective structure on $\Sigma$ means an open covering of $\Sigma$ by simply connected sets $U_i$, such that $U_i\cap U_j$ is connected and simply connected,  together with homeomorphisms $ \Phi_{i}: U_{i} \to V_i \subset \hat{\C}$, such that the overlap maps  $ \Phi_{i} \circ \Phi_{j}^{-1} : \Phi_{j}(U_{i} \cap U_{j}) \to \Phi_{i}(U_{i} \cap U_{j})$ are in  ${PSL}(2,\C)$.
 
Given the developing map $\Psi: \tilde \Sigma \to \hat\C$ from the universal covering space $\tilde \Sigma$ of $\Sigma$ into the Riemann sphere $\hat \C$, we can clearly cover $\Sigma$ by open sets $U_i$ such that each component $W$ of the lift of $ U_i$ to $\tilde \Sigma$ is homeomorphic to $U_i$ and such that the restriction $\Psi|_{W}$ is a homeomorphism  to an open set $V  \subset \hat \C$.  For each $ U_i$, pick one such component. Since any two lifts differ by a covering map, and since $U_i \cap U_j \neq \emptyset$ implies there are lifts which intersect, the overlap maps will always be in the covering group which by our construction is contained in ${PSL}(2,\C)$.
   
In terms of the projective structure, the holonomy representation $\rho: \pi_1(\Sigma) \to PSL(2,\C)$ is described as follows. A path $\gamma$ in $\Sigma$ passes through an ordered chain of charts  $U_{0}, \ldots, U_{n}$ such that $U_{i} \cap U_{i+1} \neq \varnothing$ for every $i =0, \ldots,n-1$. This gives us  the overlap maps $R_{i} = \Phi_{i} \circ \Phi_{i+1}^{-1}$ for  $i =0, \ldots, n-1$.  The sets $V_i$ and $R_{i}(V_{i+1})$ overlap in $\hat \C$ and hence the developing image of $\tilde \gamma$ in $\Chat$ passes through in order the sets $V_0, R_{0}(V_{1}), R_{0}R_1(V_{2})\ldots,  R_{0}\cdots R_{n-1} (V_n)$. If  $\gamma$ is closed, we have $U_{n} = U_{0}$ so that $V_0 = V_n$. Then, by definition, the holonomy of  the homotopy class $ [\gamma]  $ is  $\rho([\gamma]) = R_{0}\cdots R_{n-1} \in {PSL}(2,\C)$. 
 
Notice that our construction effectively takes the charts to be the maps $ P_i \to \HH$ which identify $P_i$ with the standard fundamental domain $\Delta \subset \HH$ via the map $\Phi_i: P_i \to \mathbb P$.  Strictly speaking, we should divide each  $P_i$ into two simply connected sets by cutting along its seams, so that each chart  maps to one or other of the standard ideal triangles $\Delta_0$ or $\Delta_1$. The details of how this works will be discussed in Section~\ref{sec:examples}.

As a consequence of the construction we note the following fact which underlies  the connection with the Maskit embedding (Section~\ref{sec:maskit}), and which (together with the definition of the twist in the case $q_i = 0$) proves the final statement of Theorem~\ref{thm:traceformula}. 

\begin{Lemma}\label{Parab}
Suppose that $\gamma \in \pi_1(\Sigma)$ is a loop homotopic to a pants curve $\s_{j} $. Then  $\rho(\g)$ is parabolic and $\tr \rho(\g) = \pm 2$.
\end{Lemma}

\subsection{Relation to Kra's plumbing construction}
\label{sec:plumbing}

Kra in~\cite{kra} uses essentially the above construction to manufacture surfaces by gluing  triply punctured spheres across punctures, a procedure which he calls \emph{plumbing}. Plumbing is based on so called `horocyclic coordinates' in punctured disk neighbourhoods of the punctures which have to be glued. 

Given a puncture $\e$ on a triply punctured sphere $\mathbb P$, let $\zeta: \HH \to \mathbb P$ be the natural projection, normalised so that $\e$ lifts  to $\infty \in \HH$, and so that the holonomy of the loop round $\e$ is, as above, $ \eta \mapsto \eta + 2$. Let $\D_*$ denote the punctured unit disc $\{ z \in \CC: 0 < z < 1\}$. The function $f : \HH \to \D_*$ given by  $f(\eta) = e^{i\pi \eta}$  is well defined in a neighbourhood $N$ of $\infty$ and is a homeomorphism from  an open neighbourhood of $\e$ in $\mathbb P$ to an open neighbourhood of the puncture in $\D_*$. Choosing another puncture $\e' $ of $\mathbb P$, we can further normalise so that $\e'$ lifts to $0$. Hence $f$ maps the  part of the geodesic from $\e'$ to $\e$ contained in $N$, to the interval $(0,r)$ for suitable $r>0$. These normalisations (which depend only on the choices of $\e$ and $\e'$), uniquely determine $f$. Kra calls the  natural parameter $z = f(\eta)$ in $\D_*$,  the \emph{horocyclic coordinate} of the puncture $\e$ relative to $\e'$.

Now suppose that $\hat z$ and $\hat z'$ are horocyclic coordinates for distinct punctures in distinct copies $\mathbb P_{\hat z}$ and $\mathbb P_{\hat z'}$ of $\mathbb P$. Denote the associated punctured discs by $\D_*(\hat z)$ and $ \D_*(\hat z')$. To plumb  across the  two punctures, first delete  punctured disks $   \{ 0 < \hat z <r\}$ and $  \{ 0 < \hat z' <r'\}$  from $\D_*(\hat z)$ and $ \D_*(\hat z')$ respectively. Then glue  the remaining surfaces along the annuli $$A(\hat z) =  \{ \hat z \in \D_*: r < \hat z <s\}\;\;\text{and}\;\; A(\hat z') =  \{ \hat z' \in \D_*: r' <\hat z' <s'\}$$ by the formula $\hat z \hat z'=t_{K}$. (To avoid confusion  we have written $t_K$ for Kra's parameter $t \in \CC$.) It is easy to see that this is essentially identical to our construction; the difference is simply that we implement the gluing in $\HH$ and $\HH'$ without first mapping to $\D_*(\hat z)$ and $ \D_*(\hat z')$. Our method has the advantage of having a slightly simpler formula and also of respecting the twisting around the puncture, which is lost under the map $f$.

The precise relation between our coordinates $z,z' \in \HH$ in Section~\ref{sec:moregluing} and the horocyclic coordinates $\hat z,\hat z'$ is  $$z  = f^{-1}(\hat z) = - \dfrac{i}{\pi}\log \hat z,\;\;\;\; z' = f^{-1}(\hat z')=-\dfrac{i}{\pi} \log\hat z'. $$ The relation $$\hat z \hat z'=t_K$$ translates to  $$\log t_K = \log \hat z' + \log \hat z$$ which, modulo $2\pi i \ZZ$, is exactly our relation $$ -z  + \tau = z'.$$ Hence we deduce that $$\tau =-\dfrac{i}{\pi} \log t_K.$$

\subsection{Relation to the Maskit embedding of $\Sigma$}
\label{sec:maskit}

As usual let $\PC = \{ \s_1, \ldots,\s_{\xi}\}$ be a pants decomposition of $\Sigma$. We have constructed a family of projective structures on $\Sigma$, to each of which is associated a natural holonomy representation $\rho_{\underline \tau}: \pi_1(\Sigma) \to PSL(2,\CC)$. We want to prove that our construction, for suitable values of the parameters, gives exactly the \emph{Maskit embedding} of $\Sigma$. For the definition of this embedding we follow \cite{seriesmaskit}, see also~\cite{Maskit}. Let $\RR(\Sigma)$ be the representation variety of $\pi_1(\Sigma)$, that is, the set of representations $\rho: \pi_1(\Sigma) \to PSL(2,\C)$ modulo conjugation in $PSL(2,\C)$. Let $\M \subset \cal R$ be  the subset of representations for which: 
\begin{enumerate}\renewcommand{\labelenumi}{(\roman{enumi})}
\item the group $G= \rho\left(\pi_1(\Sigma)\right)$ is discrete (Kleinian) and $\rho$ is an isomorphism, 
\item the images of $\s_i$, $i=1, \ldots, \xi$, are parabolic,
\item all components of the regular set $\Omega(G)$ are simply connected and there is exactly one invariant component  $\Omega^+(G)$,
\item the quotient $\Omega(G)/G$ has $k + 1$ components (where $k = 2g-2+n$ if $\Sigma = \Sigma_{(g,n)}$), $\Omega^+(G)/G$ is homeomorphic to $\Sigma$ and the other components are triply punctured spheres.
\end{enumerate}
In this situation, see for example~\cite{Mar} (Section 3.8), the corresponding $3$-manifold $M_{\rho} = \HH^3/G$ is topologically $\Sigma \times (0,1)$. Moreover $G$ is a geometrically finite cusp group on the boundary (in the algebraic topology) of the set of quasifuchsian representations of $\pi_1(\Sigma)$. The `top' component $\Omega^+/G $ of the conformal boundary may be identified to $\Sigma \times \{1\}$ and is homeomorphic to $\Sigma$. On the `bottom' component  $\Omega^-/G $, identified to  $\Sigma \times \{0\}$, the pants curves $\s_1,\ldots, \s_\xi$ have been pinched, making  $\Omega^-/ G$ a union of $k$ triply punctured spheres glued across punctures corresponding to the curves $\s_i$. The conformal structure on $\Omega^+/G $, together with the pinched curves $\s_1,\ldots, \s_\xi$, are the \emph{end invariants} of $M_{\rho}$ in the sense of Minsky's ending lamination theorem. Since a triply punctured sphere is rigid, the conformal structure on $\Omega^-/ G$ is fixed and independent of $\rho$, while the structure on  $\Omega^+/ G$  varies. It follows from standard Ahlfors-Bers theory, using the measurable Riemann mapping theorem (see again~\cite{Mar} Section 3.8), that there is a unique group corresponding to each possible conformal structure on $\Omega^+/ G$. Formally, the \emph{Maskit embedding} of the \Teich space of $\Sigma$ is the map $\teich(\Sigma) \to \cal R$ which sends a point  $X \in \teich(\Sigma)$ to the unique group $G \in \M$ for which $\Omega^+/ G$ has the marked conformal structure $X$.   

\begin{Proposition} Suppose that $\underline \tau \in \CC^{\xi}$ is such that the associated developing map $Dev_{\tau}: \tilde \Sigma \to \hat \C$ is an embedding.
Then  the holonomy representation $ \rho_{\underline \tau}$  is a group isomorphism and   $G = \rho_{\tau}(\pi_1(\Sigma)) \in \cal M$.
\end{Proposition}

\begin{proof}  Since the developing map $Dev: \tilde \Sigma \to \hat \C$ is an embedding,  $G = \rho_{\underline \tau} \left(\pi_1(\Sigma)\right)$ is Kleinian. By construction (see Lemma~\ref{Parab}), the holonomy of each of the curves $\s_1, \ldots, \s_\xi$ is parabolic. This proves (i) and (ii). 

The image of $Dev$ is a simply connected $G$-invariant component $\Omega^+ = Dev(\tilde \Sigma)$ of the regular set $\Omega(G)$ of $G$. Since $\Omega^+$ is $G$-invariant,   its boundary $ \partial \Omega^+$ is the limit set  $\Lambda(G)$.

Let $P \in \P$, and let $\tilde P$ be a lift of $P$ to the universal cover $\tilde \Sigma$. The boundary curves $\s_{i_1}, \s_{i_2},\s_{i_3}$ of $P$ lift, in particular, to three curves in $\partial \tilde P$ corresponding to elements $\g_{i_1}, \g_{i_2},\g_{i_3} \in \pi_1(\Sigma)$ such that $\g_{i_1}  \g_{i_2} \g_{i_3}  = id$ and such that $\rho(\g_{i_j})$ is parabolic for $j=1,2,3$. These generate a  subgroup  $\Gamma(\tilde P)$ of $SL(2,\mathbb R)$ conjugate to $\Gamma$, see Section~\ref{sec:moregluing}. Thus the limit set $\Lambda(\tilde P)$ of $\Gamma(\tilde P)$ is a round circle  $C(\tilde P) $.

Without loss of generality, fix the normalisation of $G$ such that $\infty \in \Omega^+(G)$. Since $\Omega^+(G)$ is connected, it must be contained in the component  of  $\hat \CC \setminus \Lambda(\tilde P)$ which contains $\infty$. Since $\Lambda(G) = \partial \Omega^+(G)$, we deduce that $\Lambda(G)$ is also contained   in the closure of the same component, and hence that the open disk $D( \tilde P)$ bounded by  $C(\tilde P)$ and not containing $\infty$, contains no limit points. (In the terminology of~\cite{kstop}, $\Gamma(\tilde P)$ is peripheral with peripheral disk $D( \tilde P)$.)  It follows that $D( \tilde P)$ is precisely invariant under $\Gamma(\tilde P)$ and hence that $D( \tilde P)/G = D(\tilde P)/\Gamma(\tilde P)$ is a triply punctured sphere.

Thus $\Omega(G)/G$ contains the surface $\Sigma(G)= \Omega^+(G) / G$ and the union of  $k$ triply punctured spheres $D(\tilde P)/\Gamma(\tilde P)$, one for each pair of pants in $\P$. Thus the total hyperbolic area of  $\Omega(G)/G$ is at least  $4 \pi k$. Now  Bers' area inequality~\cite{bers}, see also eg~\cite{Mat} Theorem 4.6, states that $$\rm{Area} (\Omega(G)/G) \leq 4 \pi (T-1)$$ where $T$ is the minimal number of generators of $G$, in our case $2g + b -1$. Since $k = 2g+b-2$ we have$$4 \pi (2g+b-2) \leq  {\rm Area} (\Omega(G)/G) \leq 4 \pi (T-1) = 4 \pi (2g + b -2).$$ We deduce that  $\Omega(G)$ is the disjoint union of $\Omega^+(G)$ and  the disks $D(\tilde P)$, $P \in \P$. This completes the proof of (iii) and (iv).
\end{proof}

This gives an alternative viewpoint on our main result: we are finding a formula for the leading terms in $\tau_i$ of the trace polynomials of simple curves on $\Sigma$ under the Maskit embedding of $\teich(\Sigma)$. This was the context in which the result was presented in~\cite{kstop, seriesmaskit}, see also Section~\ref{sec:opt}.
 
\section{Calculation of paths}
\label{sec:examples}

In this section we discuss how to compute the holonomy of some simple paths. We first specify a particular path joining one hexagon to the next, then we study paths contained in one pair of pants,  and finally we  compute the holonomy representations of some paths in the one holed torus and in the four times punctured sphere. 

The gluing construction in Section~\ref{sec:gluing} effectively takes the charts to be the maps which identify $P_i$ with the standard fundamental domain $\Delta \subset \HH$. Precisely, as explained above, for each  $P= P_i  \in \P$, we have a fixed homeomorphism $\Phi_i: P \to \mathbb P$ and hence a map $\zeta^{-1} \circ \Phi_i : P \to \HH$, where $\zeta$ is the  projection of $\HH$ to $\mathbb P$. Let $\Delta_0(P) = \Phi_i^{-1} \circ \zeta (\Delta_0)$ be the white hexagon in $P$.  Also let $b(P) =  \Phi_i^{-1} \circ \zeta(b_0)$ where $ b_0 = \frac{1+ i\sqrt 3}{2} \in \Delta_0$ is the barycentre of the white triangle. This will serve as a base point in $\Delta_0(P)$. 

Suppose that $\g \in \S$. Although not logically necessary, we can greatly simplify our description by arranging $\g$ in standard Penner position, so that it always passes from one pants to the next through the white hexagons $\Delta_0(P_i)$. Suppose, as in Section~\ref{sec:projstructure}, that $\g$ passes through a sequence of pants $P_{i_1}, \ldots, P_{i_n}$. We may as well assume that $\g$ starts at the base point $b(P_{i_1})$ of $P_{i_1}$. Given our identification $\Phi_{i_1}$ of $P_{i_1}$ with $\mathbb P$, there is a unique lift $\tilde b(P_{i_1}) \in \Delta_0$ and hence there is a unique lift $\tilde \g$ of $\g \cap P_{i_1}$ to $\HH$ starting at  $\tilde b(P_{i_1})$. This path exits  $ \Delta_0$ either across one of its three sides, or across that part of a horocycle which surrounds one of the three cusps $0,1,\infty$ contained in $\Delta_0$.  In the first case, the holonomy is given by the usual action of the group $\Gamma$ on $\HH$, where $\Gamma$ is the triply punctured sphere group as in Section~\ref{sec:moregluing}. (This will be explained in detail in Section~\ref{sec:onepants}.) In the second case, we have a precise description of the gluing across the boundary annuli, giving a unique way to continue $\tilde \g$ into a lift of $P_{i_2}$.  In this case we continue in a new chart in which the lift of $ P_{i_2}$ is identified with $\Delta \subset \HH$, as before. 

The following result applies to an arbitrary connected loop on $\Sigma$.
\begin{Proposition} Let $\g  \in \pi_1(\Sigma)$ and suppose that $\sum_i i(\g,\s_i) = q$. Then the trace $\tr \rho([\g])$ is a polynomial in $\tau_{1}, \ldots, \tau_{\xi}$ of maximal total degree $q$ and of maximal degree $q_{k}(\gamma) = i(\gamma,\s_k)$ in the parameter $\tau_{k}$.
\end{Proposition}

\begin{proof} Suppose the boundary $\dd_{\e} P$ of one pair of pants $P$  is glued to the boundary  $\dd_{\e'} P'$ of another pair $P'$ along  a pants curve $\s$. The map $\Phi_{P'} \Phi_{P}^{-1}: \HH_P  \to \HH_{P'}$ which glues the horocycle labelled $\e$ in $ \Delta_0(P)$ to the horocycle labelled $\e'$  in $ \Delta_0(P')$ is $\Omega_{\e'}^{-1}  T_{\tau } J \Omega_{\e}$,  where as usual the maps $\Omega_{\e}$ and $\Omega_{\e'}$ are the standard maps taking $\e,\e'$ to $\infty$. Thus with the notation of Section~\ref{sec:projstructure}, the overlap map $R=\Phi_P\Phi_{P'}^{-1}$ is 
\begin{equation}
\label{eqn:holonomy1}
\Omega_{\e}^{-1} J^{-1}T_{\tau}^{-1}\Omega_{\e'}.
\end{equation}

Any curve $\gamma \in \pi_1(\Sigma)$ which intersects the pants curves $\s_i$ in total $q$ times passes through a sequence of pants $P_{i_1} ,\ldots , P_{i_{q}} =  P_{i_1}$ and can therefore be written as a product $\prod_{j =1}^{q} \kappa_j \l_j$ where $ \kappa_j \in \pi_1(P_{i_j}; b_{i_j}) $ is a path in $P_{i_j}$ with both its endpoints in the base point $b_{i_j}= b(P_{i_j})$ and  $\l_j = \l(P_{i_j},P_{i_{j+1}})$ is a path from $b_{i_j}$ to $b_{i_{j+1}}$  across the boundary $\s_{i_j}$ between $ P_{i_j}$ and $P_{i_{j+1}}$.  

Let $\rho: \pi_1(P_{i_j}; b_{i_j} ) \to \Gamma$ be the map induced by the identification  of $P_{i_j} $ with $\Delta \subset \HH$, where $\Gamma$ is the triply punctured sphere group as in Section~\ref{sec:moregluing}. It follows that the holonomy of $\g$ is a product  
\begin{equation}
\label{eqn:holonomy}
\rho([\g]) = \prod_{j= 1}^{q } \rho(\kappa_{j})\Omega_{\e_j}^{-1}J^{-1}T_{\tau_{{i_j}}}^{-1}\Omega'_{\e_{j+1}}
\end{equation}
from which the result follows. 
\end{proof}
 
It is clear from this formula that $\tr \rho([\g]) $ is an invariant of the free homotopy class of $\g$, because changing the base point of the path $\gamma$ changes the above product by conjugation.  

We can also define the holonomy of paths with distinct endpoints in one pair of pants $P$, see Section~\ref{sec:onepants}.

\subsection{Paths between adjacent pants}
\label{sec:adjacentpaths}

Suppose that pants $P$ and $P'$ are  glued along  $\s \in \PC$. If $\dd _{\infty}P$ is glued to $\dd_{\infty}P'$,  then  there is an obvious path $\l( P, P'; \infty,\infty; 0) $ on $\Sigma_0$ from $b(P)$ to $b(P')$, namely the projection to $\Sigma_0$ of the union of the path $ z= b(P) + it, t \in [0, \Im \tau/2 + \e]$ in $\Delta_0(P)$ with the path $ z'= b(P') + it,  t \in [\Im \tau/2 + \e, 0]$ in $\Delta_0(P')$. More generally, if  $\dd _{\e}P$ is glued to $\dd_{\e'}P'$, we define  $\l(  P,  P';\e,\e';0)$ to be the path obtained by first applying the maps $\Omega_{\e}, \Omega_{\e'}$ to the segments of $\l( P, P'; \infty,\infty; 0)$ in $P,P'$ respectively.  For a general surface $\Sigma(\ttau)$ we define $\l( P, P'; \e,\e'; \ttau) = \Phi_{\ttau}(\l( P, P'; \e,\e'; 0) ) $. Note that   $\l(P,P',\s; \ttau)$ is  entirely contained in the white triangles in $P$ and $P'$. Unless needed for clarity, we refer to all these paths as $\l(P,P')$.

Referring to the gluing equation~\eqref{eqn:gluing} we see that the holonomy of $\l(P,P')$ is given by 
\begin{equation}
\label{eqn:specialpath}
\rho\left(\l(P,P'; \e,\e';\ttau)\right) =  \Omega_{\epsilon}^{-1}J^{-1}T_{\tau}^{-1}\Omega_{\epsilon'}.
\end{equation}

As already noted in Lemma~\ref{lemma:independence}, the gluing parameters $\ttau$ are independent of the direction of travel (from $P$ to $P'$ or vice versa).  From \eqref{eqn:specialpath} we have $$ \rho\left(\l(P',P; \e',\e;\ttau)\right) =   \Omega_{\epsilon'}^{-1}J^{-1}T_{\tau}^{-1}\Omega_{\epsilon}$$ so that $$ \rho(\l(P',P; \e',\e;\ttau)^{-1}) =   \Omega_{\epsilon}^{-1} T_{\tau} J\Omega_{\epsilon'}.$$ Using the identities $J^{-1} = -J$, $T_{\tau}^{-1} = T_{-\tau}$ and $T_{\tau}J = JT_{-\tau}$ this gives
\begin{equation}
\label{eqn:specialpath1}
\rho\left(\l(P',P; \e',\e;\ttau)^{-1}\right) = -\rho\left(\l(P',P; \e',\e;\ttau)\right)^{-1}, 
\end{equation}
as one would expect. That fact will be particularly important for our proof in Section~\ref{sec:mainthm}.
 
\subsection{Paths in a pair of pants}
\label{sec:onepants}

We now calculate the holonomy of the three boundary loops in one pair of pants $P$. As usual, we identify $P$ with $\mathbb P$ so that the components of its  boundary are labelled $0,1,\infty$ in some order, and  the base point is the barycentre $b_0 = \frac{1+\sqrt 3i}{2}$ of $\Delta_0$. Orient each of the three boundary curves $\dd_{\e}(\mathbb P)$, where $\e \in \{0,1,\infty\}$, consistently with the three lines $\mu_{\e}$ where, as above, $\mu_\e \subset \HH$ is the unique oriented geodesic from $\e +1$ to $\e +2$, where $\e$ is in the cyclically ordered set $\{0,1,\infty\}$, see Figure~\ref{figure4_2}. We denote by $\eta_{\e} \in \pi_1(P;b_0)$ the loop based at $b_0$ and freely homotopic to the oriented loop $\dd_{\e}(\mathbb P)$. To calculate the holonomy of $\eta_{\e}$, we begin by noting the holonomies of the three homotopically distinct  paths $\g_{\e}$, with $\e \in  \{0,1,\infty\}$, joining $b_{0}$ to $b^*_{0} = \frac{-1+\sqrt{3}i}{2}$, the barycentre of $\Delta_{1}$, see Figure~\ref{figure4_5}.   
 
\begin{figure}[hbt] 
\centering 
\includegraphics[height=6cm]{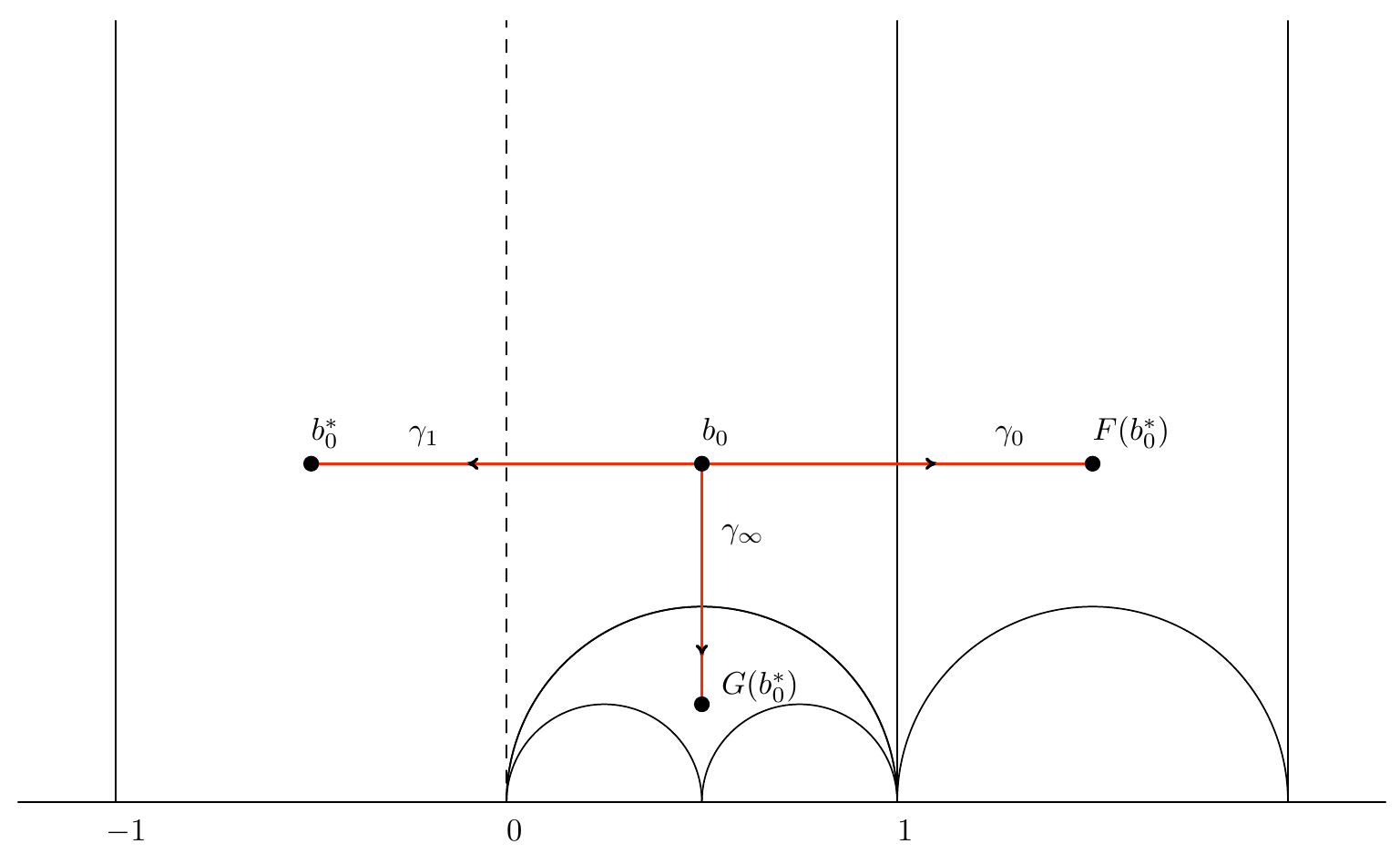} 
\caption{Paths between $b_{0}$ and $b^*_{0}$ where $F = \rho(\gamma_{0})$ and $G = \rho(\gamma_{\infty})$.}
\label{figure4_5}
\end{figure}
 
Each  path $\gamma_{\e}$ is determined by the geodesic $\mu_{\e}$ which it crosses. Thus $\gamma_{0}$ connects $b_{0}$ and $b^*_{0}$ crossing $\mu_{0}$, and so on.
The holonomies of these three paths are:
\begin{equation}
\label{eqn:basichols}
\rho(\gamma_{0}) =  	\begin{pmatrix}
	1  &  2  \\
	0 &  1  
	\end{pmatrix};\ \ 
\rho(\gamma_{1})= \Id = \begin{pmatrix}
	1  & 0 \\
	0 &  1 \end{pmatrix};\ \ 
\rho(\gamma_{\infty}) = \begin{pmatrix}
	1  &  0  \\
	2 &  1\end{pmatrix},
\end{equation}
as is clear from Figure~\ref{figure4_5}.   

\begin{figure}[hbt] 
\centering 
\includegraphics[height=6cm]{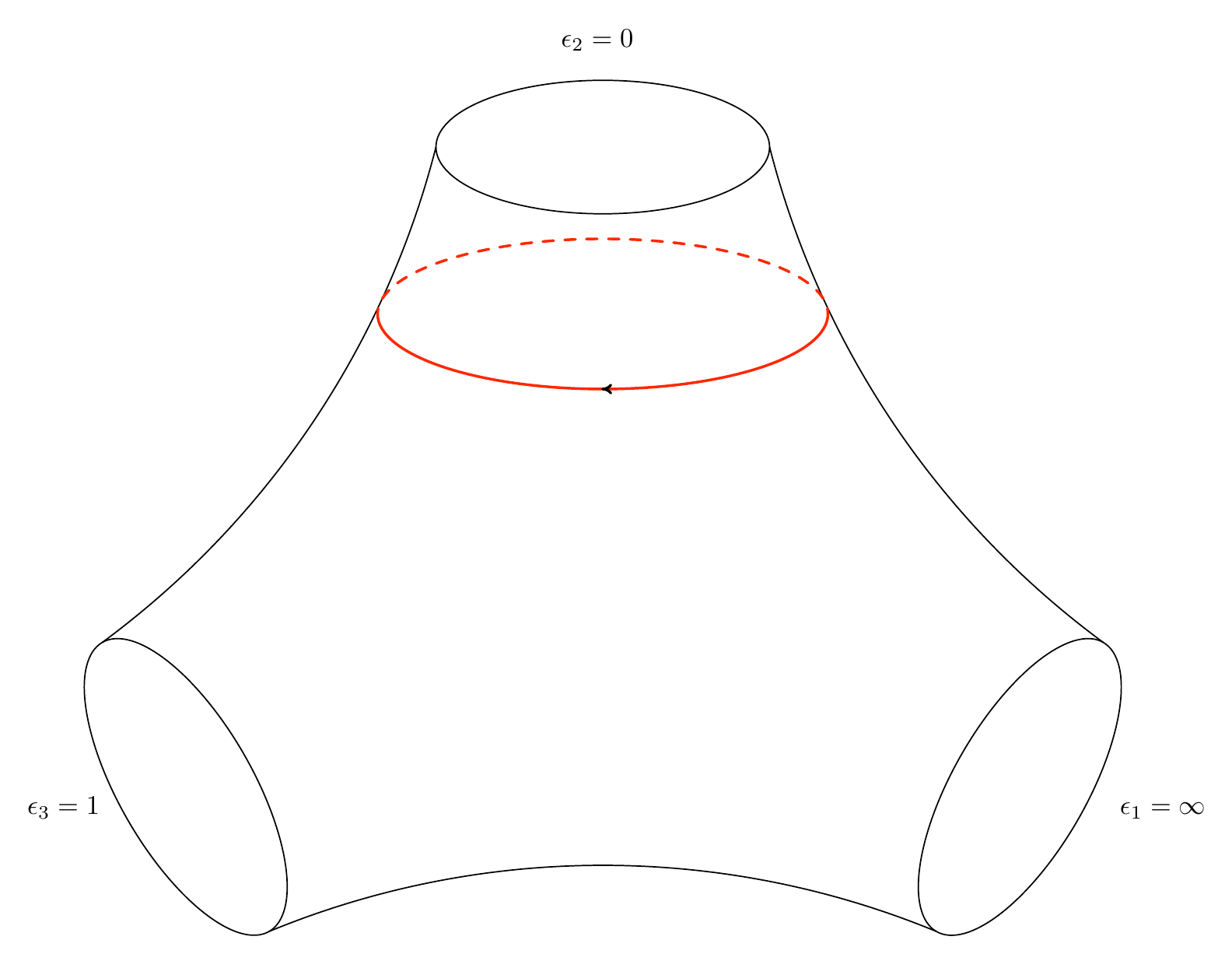} 
\caption{The loop $\eta_0$ homotopic to $\partial_{0}P$.}
\label{figure4_14}
\end{figure}

To calculate the holonomy of the loop $\eta_0$ around the boundary $\partial_{0}P$, we have to go from $b_{0}$ to $b^*_{0}$ crossing $\mu_{\infty}$ and then go from $b^*_{0}$ to $b_{0}$ crossing $\mu_{1}$. Hence, as illustrated in Figure \ref{figure4_14}, we have to go along the path $\gamma_{\infty}$ and then along the path $\gamma_{1}^{-1}$. Thus we find:
$$\rho(\eta_{0}) = \rho(\gamma_{\infty}\gamma_{1}^{-1}) =  \begin{pmatrix}
1  &  0  \\
2 &  1  
\end{pmatrix}.$$

Similarly to calculate the holonomy of $\eta_1$, we have to go from $b_{0}$ and $b^*_{0}$ crossing $\mu_{0}$ and then return from $b^*_{0}$ to $b_{0}$ crossing $\mu_{\infty}$. This means going along $\gamma_{0}$ and then  $\gamma_{\infty}^{-1}$. Thus the holonomy is:
$$\rho({\eta_1}) = \rho(\gamma_{0}\gamma_{\infty}^{-1}) = \begin{pmatrix}
1  &  2  \\
0 &  1 
\end{pmatrix}
\begin{pmatrix}
1  & 0  \\
-2 &  1 
\end{pmatrix}=  \begin{pmatrix}
-3  &  2  \\
-2 &  1  
\end{pmatrix}.$$

Finally, to calculate the holonomy of  $\eta_{\infty}$, we have to go from $b_{0}$ to $b^*_{0}$ crossing $\mu_{1}$ and return from $b^*_{0}$ to $b_{0}$ crossing $\mu_{0}$. Hence we have to go along the path $\gamma_{1}$ and then along the path $\gamma_{0}^{-1}$, so the holonomy is:
$$\rho(\eta_{\infty}) = \rho(\gamma_{1}\gamma_{0}^{-1}) = \begin{pmatrix}
1  &  -2  \\
0 &  1 
\end{pmatrix}.$$

As a check, we verify  that 
$$ \rho(\eta_0) \rho(\eta_{\infty}) \rho(\eta_1) =  \begin{pmatrix}
1  &  0  \\
2 &  1  
\end{pmatrix} \begin{pmatrix}
1  &  -2  \\
0 &  1 
\end{pmatrix}\begin{pmatrix}
-3  &  2  \\
-2 &  1  
\end{pmatrix}  = {\rm Id}$$
in accordance with the relation $  \eta_0 \eta_{\infty} \eta_1  = {\rm id} $ in $\pi_1(\mathbb P; b_0)$.

\subsection{Examples}

To conclude this section, we look at the special cases of the once punctured torus and the four times punctured sphere.     

\subsubsection{The once punctured torus}
\label{sec:opt}

\begin{figure}[hbt] 
\centering 
\includegraphics[height=9cm]{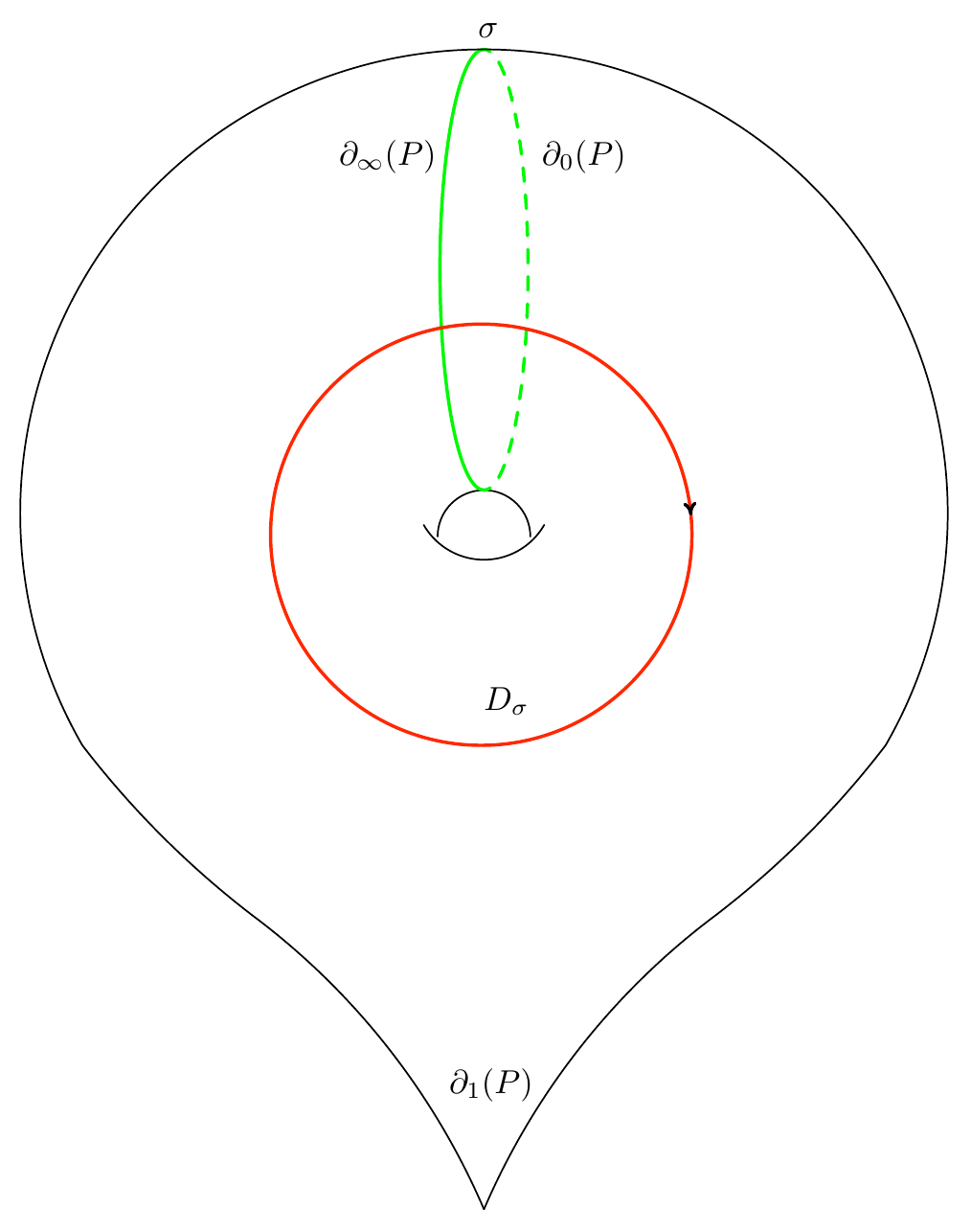} 
\caption{Plumbing for the once punctured torus.}
\label{figure4_16}
\end{figure}

The once punctured torus  $\Sigma_{1,1}$ is decomposed into one pair of pants by cutting along a single pants curve $\s$. To determine the projective structure on $\Sigma_{1,1}$ following our construction, we take a  pair of pants $P$ and glue the boundaries $\dd_{ \infty} P$ and $\dd_0 P$, so that the remaining boundary $\dd_1 P$  becomes the puncture on  $\Sigma_{1,1}$, see Figure~\ref{figure4_16}. To find $\rho: \pi_1(\Sigma_{1,1}) \to PSL(2,\C)$, it is sufficient to compute the holonomy of $\s$ and of its dual curve $D_{\s}$.

To do the gluing, take two copies of $P$ and, following the notation in Section~\ref{sec:gluing}, label the copy on the left in the figure $P$, and  that on the right, $P'$. We identify $P$ with  the standard triply punctured sphere $\mathbb P$ by the homeomorphism $\Phi: P \to \mathbb P$ so that the universal covers $\tilde P, \tilde P'$ are identified with copies $\HH, \HH'$ of the upper half plane $ \HH$ with coordinates $z,z'$ respectively. The cusps to be glued are labelled $\e  = \infty$ and $\e' = 0$.  We first apply the standard symmetries $\Omega_{\e }, \Omega_{\e'}$ which carry  $\epsilon = \infty$ and $\epsilon' = 0$ to $\infty$.  Referring to~\eqref{eqn:standardsymmetries}, we see that  $\Omega_{\infty} (z) = z$ and   $\Omega_0(z') = 1 -  \frac{1}{z'}$.

According to the choices made in Section~\ref{sec:marking}, the dual curve $D_{\s}$ to $\s$ is  the curve $\l(P,P',\s;\tau)$ joining $b(P) \in P $ to $b(P') \in P'$. By~\eqref{eqn:holonomy} in Section~\ref{sec:examples} and by the formulae~\eqref{eqn:standardsymmetries} for the standard symmetries, we have: 
\begin{align*}
	\rho(D_{\s}) &= \Omega_{\infty}^{-1}J^{-1}T_{\tau}^{-1}\Omega_{0}\\
	&= 
\begin{pmatrix}
i  &  0 \\
0 &  -i  
\end{pmatrix}\begin{pmatrix}
1  &  -\tau \\
0 &  1 
\end{pmatrix}\begin{pmatrix}
1  &  -1  \\
1 &  0
\end{pmatrix}
  = -i\begin{pmatrix}
\tau-1  &  1 \\
1 &  0 
\end{pmatrix}.
\end{align*}
Since clearly $\rho(\s) = \begin{pmatrix}
1  &  2 \\
0 &  1  
\end{pmatrix}$, this is enough to specify the representation $\rho: \pi_1(\Sigma) \to \mathrm{PSL}(2,\C)$.

\medskip

The original motivation for studying the representations in this paper came from studying the Maskit embedding of the once punctured torus, see~\cite{kstop} and Section~\ref{sec:maskit}. The Maskit embedding for $\Sigma_{1,1}$ is described in~\cite{kstop} as the representation $\rho': \pi_1(\Sigma_{1,1}) \to PSL(2,\CC)$ given by 
$$\rho'(\s) =   \begin{pmatrix}
1 &2 \\
0&  1  
\end{pmatrix} \ \ \rm{and} \ \ \rho'(D_{\sigma})  = -i \begin{pmatrix}
\mu & 1 \\
1 &  0   
\end{pmatrix}.$$ 
This agrees with the above formula setting $\mu = \tau -1$.

\subsubsection{The four holed sphere $\Sigma_{0,4}$}
\label{sec:fourholed}

We decompose $\Sigma_{0,4}$ into two pairs of pants $P$ and $P'$ by cutting along the curve $\s$, and label the boundary components as shown in Figure~\ref{figure4_17}, so that the boundaries to be glued are both labelled $\infty$. In the figure, $P$ is the upper of the two pants and $P'$  the lower. We shall calculate the holonomy of the dual $D_{\s}$  to $\s$  in two different ways, first in the standard Penner position and secondly in the symmetrical DT-position. As it is to be expected, the two calculations give the same result.

\medskip

\paragraph{\textbf{The loop $D_{\s}$ in Penner standard position.}}

If we put the loop $D_\s$ in Penner standard position, as illustrated in Figure~\ref{fig:dualPenner}, and as described in Section~\ref{sec:pennertwist}, we see that it is the concatenation of the paths:
\begin{enumerate}
\item $\l(P,P';\infty,\infty; \tau)$ from $b_0(P)$ to $b_0(P')$;
\item $\eta_{0}(P')$;  
\item $\eta_{\infty}(P')$;
\item $\l(P',P;\infty,\infty; \tau)$ from $b_0(P')$ to $b_0(P')$;
\item $\eta_{0}^{-1}(P)$.
\end{enumerate}

Thus using the calculations in Sections~\ref{sec:adjacentpaths} and~\ref{sec:onepants}, we have
\begin{align*}
\rho(D_{\s}) & = \rho\left(\l(P,P'; \infty,\infty; \tau) \cdot \eta_{0}(P') \cdot \eta_{\infty}(P') \cdot \l(P',P; \infty,\infty; \tau) \cdot \eta_{0}^{-1}(P) \right)\\
             & = \begin{pmatrix} i &-i\tau \\0 & -i \end{pmatrix} \cdot \begin{pmatrix} 1 & 0 \\ 2 & 1 \end{pmatrix} \cdot \begin{pmatrix} 1 & -2 \\ 0 & 1 \end{pmatrix} \cdot \begin{pmatrix} -i & i\tau \\0 & i \end{pmatrix} \cdot \begin{pmatrix} 1 & 0 \\ -2 &1 \end{pmatrix}\\
			& = \begin{pmatrix} -4\tau^2+6\tau-3 & 2\tau^2-4\tau+2\\-4\tau+4 & 2\tau-3 \end{pmatrix}  
\end{align*}
giving  $$\tr \rho(D_{\s}) = -4\tau^2+8\tau-6.$$ Now  $q(D_{\s}) = 2$ and $p(D_{\s}) = 0$ (see Section~\ref{sec:dylantwist}), and the number $h$  of $scc$-arcs  in $D_{\s}$ is $2$. Thus Theorem~\ref{thm:traceformula} predicts that $$ \tr \rho(D_{\s}) =  \pm i^2 2^2(\tau + (0-2) /2 )^2+R $$ where $R$ represents terms of degree at most $0$ in $\tau$,  in accordance with the computation above.
 
\medskip

\paragraph{\textbf{The loop $D_{\s}$ in symmetrical DT-position.}}

\begin{figure}[hbt] 
\centering 
\includegraphics[height=8.5cm]{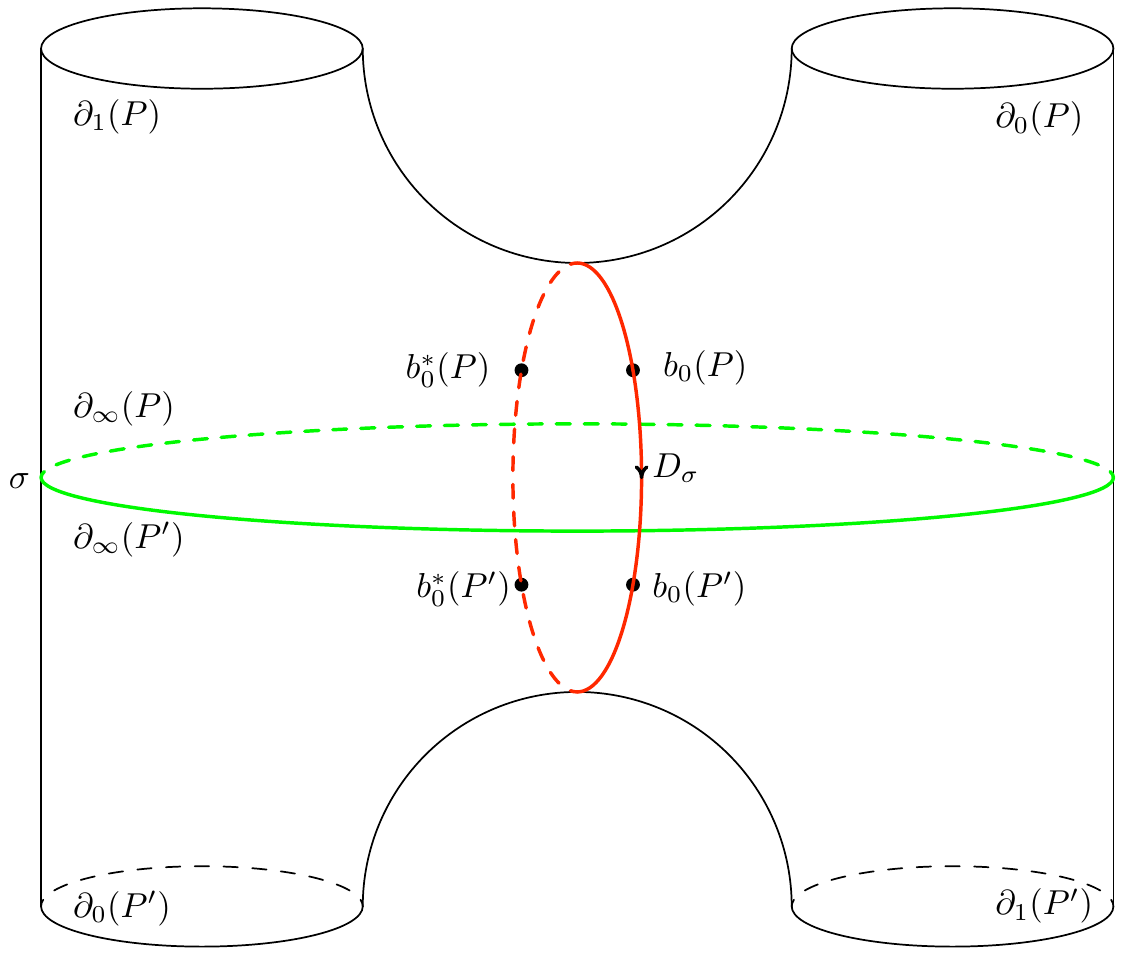} 
\caption{The loop $D_\s$ in its symmetrical DT-position in the four holed sphere.} 
\label{figure4_17}
\end{figure}

As usual we take as base points the barycenters $b_0(P)$ and $b_0^*(P)$ of the `white' and the `black' hexagons respectively in $P$ and the same base points $b_0(P')$ and $b_0^*(P')$ in $P'$. Also denote $\l^*(P,P';\infty,\infty; \tau)$ the path $R(\l(P,P';\infty,\infty; \tau))$ from $b_0^*(P)$ to $b_0^*(P')$ through the black hexagons, where $R$ is the orientation reversing symmetry of $\Sigma (\tau)$ as in Section~\ref{sec:dylantwist}.

From Figure~\ref{figure4_17}, we see that $D_\s$ is the concatenation of the paths:
\begin{enumerate}
\item $\l(P,P';\infty,\infty; \tau)$ from $b_0(P)$ to $b_0(P')$;
\item $\g_{\infty}(P')$ in $P'$ from $b_0(P')$ to $b_0^*(P')$;  
\item $\l^*(P',P;\infty,\infty; \tau)$ from $b_0^*(P')$ to $b_0^*(P')$;
\item $\g_{\infty}^{-1}(P)$ in $P$ from $b_0^*(P)$ to $b_0(P)$.
\end{enumerate}
Thus $$\rho(D_{\s} ) =\rho\left(\l(P,P',\s; \tau)\right) \cdot \rho\left(\g_{\infty}(P)\right) \cdot \rho\left(\l^*(P',P,\s; \tau)\right) \cdot \rho\left(\g_{\infty}^{-1}(P)\right).$$ Following  Remark~\ref{Rem:black-white} we have $$\l^*(P',P;\infty,\infty; \tau) = \l(P',P;\infty,\infty; \tau-2) = \l(P,P';\infty,\infty; \tau-2)^{-1}$$ so that, from Section~\ref{sec:adjacentpaths}, we have $\rho\left(\l^*(P',P;\infty,\infty; \tau)\right) = \begin{pmatrix} -i & i(\tau-2) \\0 & i \end{pmatrix} $. Thus  referring also to Section~\ref{sec:onepants} we see 
\begin{align*}
   \rho(D_{\s}) & = \begin{pmatrix} i &-i\tau \\0 & -i \end{pmatrix} \cdot  \begin{pmatrix} 1 & 0 \\ 2 & 1 \end{pmatrix} \cdot \begin{pmatrix} -i & i(\tau-2) \\0 & i \end{pmatrix} \cdot  \begin{pmatrix} 1 &0 \\ -2 &1 \end{pmatrix}\\
                 & = \begin{pmatrix} -4\tau^2+6\tau-3 & 2\tau^2-4\tau+2\\-4\tau+4 & 2\tau-3 \end{pmatrix}.
\end{align*}
Hence $\tr\left( \rho(D_{\s}) \right)= -4\tau^{2}+8\tau-6$  as before.

\section{Proof of Theorem~\ref{thm:traceformula}} 
\label{sec:mainthm}

In this final section we prove Theorem~\ref{thm:traceformula}. Our method is to show that the  product of matrices forming the holonomy  always takes a special form and then give an inductive proof. 

First consider the holonomy representation of a typical path. Let $\g \in \S_0$ be a simple loop on $\Sigma$.  We suppose $\gamma$ is in Penner standard position, so that it always cuts $\s_{i_j}$ in the arc $w_{i_j}$. Starting from the basepoint in some pants $P$, it crosses, in order, pants curves $\s_{i_j}$, $j = 1, \ldots, q(\g)$. If the boundaries glued across $\s_{i_j}$ are $\partial_{\e}P,\partial_{\e'}P'$, then, by equation~\eqref{eqn:holonomy1}, the contribution to the holonomy product $\rho(\gamma)$ is $$\Omega_{\e}^{-1}J^{-1}T_{\tau_{{i_j}}}^{-1}\Omega'_{\e'}$$ where $\tau_{{i_j}}=\tau_i$ whenever $\s_{i_j} = \s_i \in \PC$.

A single positive twist  around $\dd_{\e}P$ immediately before this boundary component contributes $\rho(\eta_{\e}^{-1}) = \Omega_{\e}^{-1} \rho(\eta_{\infty}^{-1})\Omega_{\e}$ (because $\eta_{\e}$ twists in the \emph{negative} direction round $\dd_{\e}P$, see Figure~\ref{figure4_14}), while a single positive twist around $\dd_{\e'}P'$ after the crossing contributes $\rho(\eta_{\e}) = \Omega_{\e'}^{-1} \rho(\eta_{\infty})\Omega_{\e'}$. Thus if, in general, $\gamma$ twists $\alpha_{j}$ times  around $\partial_{\e}P = \s_{i_j}$ immediately before the crossing and $\beta_{j}$ times after, the total contribution to the holonomy is
\begin{equation}
\label{eqn:hol2}
\Omega_{\e}^{-1}\rho(\eta_{\infty})^{-\alpha_j}J^{-1}T_{\tau_{{i}}}^{-1}\rho(\eta_{\infty})^ {\beta_j}\Omega'_{\e'},
\end{equation}
where $\sigma_{i_j} = \sigma_i \in \PC$.

From Sections~\ref{sec:moregluing} and~\ref{sec:onepants} we have  $$J^{-1}T_{\tau}^{-1} = \begin{pmatrix} i & -i\tau \\  0 & -i \end{pmatrix} \ \ \rm{and} \ \ \rho(\eta_{\infty}) = \begin{pmatrix} 1 & -2 \\  0 & 1 \end{pmatrix}.$$ 
For  variables $X,Y$, write $A_X =  \begin{pmatrix} 1 &X \\  0 & -1 \end{pmatrix}$ and $B_Y = \begin{pmatrix}  1 & Y \\ 0 & 1\end{pmatrix}$. We calculate  $$ \rho(\eta_{\infty})^{-\alpha_j}J^{-1}T_{\tau_{{i}}}^{-1}\rho(\eta_{\infty})^ {\beta_j} = iA_{X_j}$$ with $X_j = -\tau_{i}+2\alpha_j + 2 \beta_j$, from which we note in particular that, as expected, which side of the crossing the twists occur makes no difference to the final product.

\begin{Proposition}  
\label{prop:holproduct} 
\begin{flushleft}\end{flushleft}
\begin{enumerate} \renewcommand{\labelenumi}{(\roman{enumi})}
\item   Suppose that  $\gamma$ contains no $scc$-arcs. Then   $\rho(\gamma)$ is of the form $\pm i^q\Pi_{i=1}^{q}A_{X_j}\Omega_{i_j}$, where $ \Omega_{i_j} = \Omega_{0} \ \ {\rm or}  \ \ \Omega_{1}$ for all $j$. If  the term $A_{X_j}$  corresponds to the crossing of a pants curve $\s_{i_j} = \s_i$, with   $\alpha_j$ twists before the crossing and $\beta_j$ after, then $X_j = -\tau_{i}+2\alpha_j + 2 \beta_j$.
\medskip 
\item  If $\gamma$ contains $scc$-arcs, then $\rho(\gamma)$ takes the same form as above, with an extra term $$A_{X_j} \Omega_1 B_{\pm 2}\Omega_0  A_{X_j}$$ inserted  for each $scc$-arc which crosses $\s_{i_j}$ twice in succession. 
\medskip  
\item In all cases, the total P-twist of  $\gamma $ about $\s_i \in \PC$ is  $\hat p_i (\gamma) = \sum_{\s_{i_j} = \s_i} (\alpha_j +  \beta_j)$.
\end{enumerate}
\end{Proposition}

\begin{proof} As  computed above we have $$\rho(\eta_{\infty})^{-\alpha}J^{-1}T_{\tau_{{}}}^{-1}\rho(\eta_{\infty})^ {\beta}
= i\begin{pmatrix} 1 & -\tau+2\alpha + 2 \beta \\  0 & -1 
\end{pmatrix} = iA_{-\tau+2\alpha + 2 \beta}.$$ It follows that the holonomy $\rho(\gamma)$ is a concatenation of $q$   terms  of the form $\Omega^{-1}_{\e}A_{-\tau+2\alpha + 2 \beta}\Omega_{\e'}$, one for each crossing of a pants curve $\s_{i_j}$. If $\gamma$ contains no $scc$-arcs,  then it enters and leaves each pants $P$ across distinct boundary components, say $\partial_{\e_1} P$ and   $\partial_{\e_2} P$ respectively. Then the corresponding  adjacent terms in the concatenated product are then $$\ldots   \Omega_{\e_1}  \Omega_{\e_2}^{-1}\ldots$$ where $\e_1 \neq\e_2$, from which  (i) easily follows.

We also note that regardless of how the twists are organised before or after the crossings, the sum $\sum_{j = 1}^{q_{i}(\g)} (\alpha_{j}+\beta_{j})$ of coefficients in terms corresponding to crossings of the pants curve $\s_i$ is equal to $\hat p_{i}(\g)$, the $i$-th P-twisting number of the curve $\g$ with respect to the pants decomposition $\mathcal{PC}$.  This proves (iii).

Now suppose that $\gamma$ contains some $scc$-arcs. Suppose that $\gamma$ cuts a curve $\s_{i_j}$ twice in succession entering and leaving a pants $P$ across the boundary  $\partial_{\infty} P$.  Since $\gamma$  is  in P-form, after crossing $\partial_{\infty} P$ it goes once around $\partial_{0} P$ in either the positive or negative direction and then returns to $\partial_{\infty} P$, see Figure~\ref{fig:dualPenner}.   Since $\g$ is simple,  the twisting around $\s_{i_j}$ is the same on the inward and the outward journeys. The term in the holonomy is therefore $$A_{X_j} \rho(\eta_0)^{\pm 1} A_{X_j} = A_{X_j} \Omega_0^{-1} \rho(\eta_{\infty} )^{\mp 1}\Omega_0  A_{X_j} = A_{X_j} \Omega_0^{-1} B_{\pm 2}\Omega_0  A_{X_j}$$ as claimed. 
 
If  more generally $\gamma$ enters and leaves $P$ across $\partial_{\e} P$, then this entire expression is multiplied on the left by $\Omega_{\e}^{-1}$ and on the right by $ \Omega_{\e}$. By the same discussion as in (i), this leaves the form of the holonomy product unchanged. The contribution to the twist about $\s_{i_j}$ is calculated as before.
\end{proof}

We are now ready for our inductive proof of Theorem~\ref{thm:traceformula}. Suppose first that $\gamma \in \S_0$ contains no $scc$-arcs. If $$\rho(\gamma) = \Pi_{i=1}^{q}A_{X_j}\Omega_{i_j}$$ define $ X^*_j = X_j +h(\Omega_{i_j})  +k(\Omega_{i_{j-1}})$ where $\Omega_{i_0} :=\Omega_{i_q}$ and $$h(\Omega_{i_j}) = \begin{cases} 1 & \text{if\;} i_j=0\\
                                  0 & \text{otherwise\;}
	                \end{cases}\;\;\;\; \;\;\text{and\;} \;\;\;\;\;	k(\Omega_{i_j}) = \begin{cases} 0 & \text{if\;} i_j=0\\
								                                                    1 & \text{otherwise}.
									                \end{cases}$$ Thus:
\begin{align*}
&\Omega_0 A_X \Omega_0 \rightarrow X^* = X+1 \\
&\Omega_0 A_X \Omega_1 \rightarrow X^* = X \\
&\Omega_1 A_X \Omega_0 \rightarrow X^* = X+2 \\
&\Omega_1 A_X \Omega_1 \rightarrow X^* = X+1.\\
\end{align*}

\begin{Remark}\label{rem:inverseinvce} {\rm  Replacing $\rho(\gamma)$ by $\rho(\gamma)^{-1}$ leaves the occurrences of the above blocks   unchanged. The entire matrix product is multiplied by $(-1)^q$. This is because $A_X^{-1} = -A_X$ and,  for example, $$(\Omega_0 A_X \Omega_1)^{-1} = \Omega_1^{-1} A_X^{-1} \Omega_0^{-1} = -\Omega_0 A_X \Omega_1.$$}
\end{Remark}

Now given the path of some $\g \in \S_0$, consider a crossing   for which $\s_{i_j} = \s_i $. Let $\Omega_{i_{j-1}} A_{X_j} \Omega_{i_j}$ be the corresponding terms in $\rho(\gamma)$, (where $\Omega_{i_{j-1}} $ is associated to the crossing of the previous pants curve $\sigma_{i_{j-1}}$). Let $p_j, \hat p_j$ be the respective contributions from this $j^{th}$ crossing to the DT- and P-twist coordinates of $\gamma$, so that the total twists $p_i, \hat p_i$ about $\s_i$ are obtained by summing over all crossings for which $\s_{i_j} = \s_i$: $p_i = \sum_{\s_{i_j} = \s_i} p_j$ and likewise $\hat p_i = \sum_{\s_{i_j}  = \s_i }\hat p_j$. Also define $p_j^*$ according to the same rule as $X^*$ above, in other words $p_j^* = p_j +h(\Omega_{i_j})  +k(\Omega_{i_{j-1}})$. Finally define  $p^*_i = \sum_{\s_{i_j}= \s_i}  p_j$. We have:
 
\begin{Lemma}\label{lem:trace1} 
 Suppose that $\gamma$ contains no $scc$-arcs and as usual let $p_i, \hat p_i$ be the DT- and P-twists of $\gamma$ about the pants curve $\s_i \in \PC$, while $q_i = i(\s_i, \gamma)$. Then $2\hat p_i =  p^*_i -q_i$. 
\end{Lemma}
 
\begin{proof} This is verified using Theorem~\ref{thm:twistrelation}, together with the fact that $\gamma$ is assumed to be in P-standard form. 

Consider a crossing   for which $\s_{i_j} = \s_i $ with corresponding term $\Omega_{i_{j-1}} A_{X_j} \Omega_{i_j}$ in $\rho(\gamma)$. With  $p_j, \hat p_j, p_j^*$ defined as above, we claim that $2\hat p_j = p^*_j  -1$.
 
Suppose for example that the relevant term in the holonomy is $\Omega_0 A_X \Omega_0$, so that by definition, $p_j^* = p_j +1$. Without loss of generality, we may suppose that $\s_{i_j} $ is the gluing of $\partial_{\infty}P$ to $\partial_{\infty}P'$ as shown in Figure~\ref{Figure1_7}. The first $\Omega_0$ means that there is an arc from $D$ to $E$, and the second $\Omega_0$ means there is an arc from $E$ to $A$. The formula of Theorem~\ref{thm:twistrelation} therefore gives a contribution $\hat p_j = (p_j + 0 + 1 -1)/2 = p^*_j -1$. 

Checking the various other possible combinations and then summing over all $q_i$ terms $\s_{i_j}$ with $\s_{i_j} = \s_i $ gives the result.
\end{proof}

In the case of no $scc$-arcs, Theorem~\ref{thm:traceformula} is an immediate corollary of this lemma and the following proposition:

\begin{Proposition}\label{thm:trace1} 
Suppose that $\gamma$ contains no $scc$-arcs, then $$\tr \left( \rho(\g)\right)  = \tr \left( \Pi_{j=1}^{q} A_{X_j} \Omega_{i_j}\right) = \pm \left(\Pi_{j=1}^{q}   X^*_j\right) + R$$  where $R$ denotes terms of degree at most $q-2$ in the $X_j$.
\end{Proposition}

\begin{proof}[Proof of Theorem~\ref{thm:traceformula}] (No $scc$-arcs case.)  
By Proposition~\ref{prop:holproduct}, if  $\sigma_{i_j} = \sigma_i$ then $ X^*_{j}  = -\tau_{i}-2\alpha_j - 2 \beta_j$. There are $q_i(\gamma)$ such terms $X_j$. Thus the top order term of $\tr \rho(\gamma)$ is  $\tau_1^{q_1} \ldots \tau_\xi^{q_\xi}$, with coefficient $\pm i^q$, in accordance with the result of Theorem~\ref{thm:traceformula}.

Multiplying out  $\Pi_{j=1}^{q}   X^*_j $ and taking account of the $q_i$ terms $X_j$ with $\sigma_{i_j} = \sigma_i$, we see from Proposition~\ref{thm:trace1} that  the coefficient of $  \tau_1^{q_1} \ldots \tau_i^{q_i -1} \ldots \tau_\xi^{q_\xi}$ is $$\sum_{\sigma_{i_j} = \sigma_i}  2\alpha_j + 2 \beta_j = 2 \hat p_i.$$ From Lemma~\ref{lem:trace1} we  have $2 \hat p_i =  p^*_i - q_i$, which is exactly the coefficient in  Theorem~\ref{thm:traceformula}.
\end{proof}

\begin{proof} [Proof of Proposition~\ref{thm:trace1}.]   
We prove this by induction on the length $q$ of the product $ \Pi_{j=1}^{q} A_{X_j} \Omega_{i_j}$. If $q=1$, with respect to the cyclic ordering we see either the block $ \Omega_0 A_X \Omega_0$ or $ \Omega_1 A_X \Omega_1$, so that $X^* = X+1$. In both cases we check directly that $\tr A_X \Omega_0 = \tr A_X  \Omega_1 = X+1$. 

\medskip

The case $q=2$ corresponds to a product $A_{X_1}  \Omega_\e A_{X_2}  \Omega_{\e'}$. Hence there are four possibilities corresponding to $\e = \pm 1$ and $\e' = \pm 1$. These cases can be checked either by multiplying out or by using the trace identity
\begin{equation}\label{eqn:grandfather}
\tr (AB) =  \tr (A)\tr (B) -  \tr (AB^{-1}).
\end{equation}
For example, if $\e = \e' = 0$, then
\begin{eqnarray*}
\tr A_{X_1}  \Omega_0 A_{X_2}  \Omega_0&=& \tr (A_{X_1}  \Omega_0) \tr (A_{X_2}  \Omega_0)+ \tr A_{X_1}A_{X_2} \\&=&   (X_1+1)(X_2+1) +2 = X^*_1X^*_2 + 2,
\end{eqnarray*} 
where we used the relation $A_X^{-1} = -A_X$. 

If  $\e = 0, \e' = 1$ then $$\tr \left(A_{X_1}  \Omega_0 A_{X_2}  \Omega_1\right)= \tr (A_{X_1}  \Omega_0) \tr (A_{X_2}  \Omega_1)- \tr \left(A_{X_1} \Omega_0 \Omega_1^{-1}A_{X_2}^{-1}\right). $$ The first term on the right hand side is $(X_1+1)(X_2+1)$ and the last term reduces to $$-\tr A_{X_2}^{-1}A_{X_1} \Omega_1 = -X_2+X_1-1.$$  Hence $$  \tr A_{X_1}  \Omega_0 A_{X_2}  \Omega_1 = X_1X_2 +2X_2 +2 = (X_1+2)X_2 +2 =  X^*_1 X^*_2 +2 .$$ The other two cases with $q=2$ are similar (or can be obtained from these by replacing $\gamma$ with $\gamma^{-1}$).

\medskip

Now we do the induction step. Suppose the result true for all products of length less than $q$. We split into three cases.

\medskip

\noindent \emph {Case (i): $\Omega_0$ appears $3$ times consecutively.}

After cyclic permutation the product is of the form $$ A_{X_1} \Omega_0 A_{X_2} \Omega_0 A_{X_3} \Omega_0 \ldots \Omega_{i_q}.$$ We will apply~\eqref{eqn:grandfather}, splitting the product  as $$(A_{X_2} \Omega_0) \times (A_{X_3} \Omega_0 \ldots \Omega_{i_q}A_{X_1} \Omega_0).$$  Considering the first term of this split product alone, $A_{X_2}$ is still preceded and followed by $\Omega_0$. Likewise, taking the second term alone, $A_{X_1}$ and $A_{X_3}$ are still preceded and followed by the same values $\Omega_i$ and nothing else has changed. Thus the induction hypothesis gives $$ \tr A_{X_2} \Omega_0 =   X^*_2 \;\;\; \text{and } \;\;\; \tr A_{X_3} \Omega_0 \ldots \Omega_{i_q}A_{X_1} \Omega_0=  X^*_3 \ldots \ X^*_q   X^*_1.$$

Now consider the remaining term coming from ~\eqref{eqn:grandfather}: $$\tr [A_{X_2} \Omega_0(A_{X_3} \Omega_0 \ldots \Omega_{i_q}A_{X_1} \Omega_0)^{-1}] = \tr [A_{X_2} A_{X_1}^{-1}\Omega_{i_q}^{-1} \ldots A_{X_3}^{-1}].$$ Cyclically permuting, the three terms $A_{X_3}^{-1}, A_{X_2}, A_{X_1}^{-1}$ combine to give a single term $A_{X_3 + X_2+ X_1}$, so that the trace has degree at most $q-2$ in the variables $X_3 + X_2+ X_1,X_4,\ldots,X_q$. Putting all this together proves the claim.

\medskip

\noindent \emph{Case (ii): $\Omega_0$ appears at most $2$ times consecutively.} 

Suppose first $q \geq 4$. Thus after cyclic permutation the product is of the form: $$ A_{X_1} \Omega_0 A_{X_2} \Omega_0 A_{X_3} \Omega_1 A_{X_4} \ldots \Omega_{1}.$$ We apply~\eqref{eqn:grandfather}  splitting as $$(A_{X_1} \Omega_0 A_{X_2} \Omega_0 A_{X_3} \Omega_1 ) \times (A_{X_4} \ldots A_{X_q} \Omega_1).$$ Taking each of these subproducts separately, we see that again the terms $\Omega_i$ preceding and following each $A_X$ are unchanged. So the induction hypothesis gives $$\tr (A_{X_1} \Omega_0 A_{X_2} \Omega_0 A_{X_3} \Omega_1 ) =  X^*_1 X^*_2 X^*_3$$ and $$\tr (A_{X_4} \ldots A_{X_q} \Omega_1) = X^*_4 \ldots X^*_q.$$ Moreover we note $$A_{X_1} \Omega_0 A_{X_2} \Omega_0 A_{X_3} \Omega_1 \Omega_1^{-1}  A_{X_q}^{-1}\ldots A_{X_4}^{-1}$$ is of degree at most $q-2$ in the variables $X_4+X_1, X_2, X_3+X_q, X_5,\ldots,X_{q-1}$. The result follows.

The case $q =3$ is dealt with by splitting $$A_{X_1} \Omega_0 A_{X_2} \Omega_0 A_{X_3} \Omega_1\;\;\;\;\; \text{as} \;\;\;\;\; (A_{X_3} \Omega_1 A_{X_1} \Omega_0) \times A_{X_2} \Omega_0,$$ using the previously considered case $q=2$, and noting that $$ A_{X_3} \Omega_1 A_{X_1} \Omega_0 \Omega_0^{-1} A_{X_2}{-1} $$ has degree $1$ in the variable $X_1+X_2+X_3$.

\medskip

\noindent \emph{Case (iii): $\Omega_0$ and $\Omega_1$ appear alternately.}

In this case we split $$ A_{X_1} \Omega_0 A_{X_2} \Omega_1 A_{X_3} \Omega_0 A_{X_4} \ldots \Omega_{1} \;\;\; \text{as} \;\;\; (A_{X_1} \Omega_0 A_{X_2} \Omega_1)\times ( A_{X_3} \Omega_0 A_{X_4} \ldots \Omega_{1})$$ and the argument proceeds in a similar way to that before.
\end{proof} 

\medskip

Now we add in the effect of having $scc$-arcs, that is we deal with the case $h>0$. 

\begin{Theorem}\label{thm:trace2} 
Suppose that a matrix product of the form in Theorem~\ref{thm:trace1} is modified by the insertion of $h$ blocks $ A_{X_j} \Omega_0^{-1} B_{Y_r}\Omega_0  A_{X_j} $, $r = 1,\ldots,h$ for variables $Y_r \in \CC$. Then its trace is $$ \pm \left(\Pi_{j=1}^{h} Y_k\right) \left(\Pi_{j=1}^{q} X^*_j\right) + R$$  where $R$ denotes terms of degree at most $q-2$ in the $X_j$.
\end{Theorem}

\begin{proof}
We  first check the case $q=1, h=1$ by hand. (Note that such a  block cannot be the holonomy matrix of a simple closed curve.) We have: $$A_{X} \Omega_1 B_{Y}\Omega_0 =  \begin{pmatrix}  X(1+Y) -(X+1)& * \\ * & 1\end{pmatrix} ,$$ hence $\tr A_{X} \Omega_1 B_{Y}\Omega_0 = XY$. Since the term $\Omega_0 A_X\Omega_1$ contributes the factor $X$  and the term $B_Y$ contributes $Y$, this fulfills our hypothesis.

\medskip

Now work by induction on $h$.  Suppose the result holds for products $$  \Omega_u A_{X_1} \Omega_{i_1} A_{X_2} \Omega_{i_2} \ldots A_{X_s}$$ containing at most $h-1$ terms of the form $B_{Y_r}$ and consider a product $$  \Omega_u A_{X_1} \Omega_{i_1} A_{X_2} \Omega_{i_2} \ldots \Omega_v A_{X} \Omega_1 B_{Y_h}\Omega_0 A_{X}.$$ 

There are four possible cases:

$u=1,v=0;\;\; u=1,v=1;\;\;u=0,v=0;\;\;u=0,v=1$.

\medskip

\noindent \emph{Case $u=1,v=0$.} Consider the extra contribution to the trace resulting from the additional block $A_{X} \Omega_1 B_{Y_h}\Omega_0 A_{X}$. The first occurrence of $A_X$ appears in a block $\Omega_0 A_{X}  \Omega_1$ which, according to what we want to prove, should contribute a factor $X$. Likewise the block  $\Omega_0 A_{X}  \Omega_1$  containing the second occurrence of $A_X$ should contribute $X$, and the term $B_Y$ should contribute $Y$. Thus it is sufficient to show that 
\begin{eqnarray*} 
\tr  ( \Omega_1 A_{X_1} \Omega_{i_1} A_{X_2} \Omega_{i_2} \ldots  A_{X_s} \Omega_0 A_{X} \Omega_1 B_{Y}\Omega_0 A_{X}) &=& \\  \pm  X^2Y\tr  ( \Omega_1 A_{X_1} \Omega_{i_1} A_{X_2} \Omega_{i_2} \ldots  A_{X_s} \Omega_0 A_X ) + R 
\end{eqnarray*} 
where $R$ denotes terms of total degree at most $2$ less then the first term in the $X_j$.

Splitting the product as $$(\Omega_1 A_{X_1} \Omega_{i_1} A_{X_2} \Omega_{i_2} \ldots \Omega_0 A_{X} )\times (\Omega_1 B_{Y}\Omega_0 A_{X})$$ and using~\eqref{eqn:grandfather}, we see that the second factor contributes $XY$ and the first factor, containing the sequence $\Omega_0 A_X  \Omega_1$, contributes $X$. The remaining term $$ ( \Omega_1 A_{X_1} \Omega_{i_1} A_{X_2} \Omega_{i_2} \ldots \Omega_0 A_{X} )\times (\Omega_1 B_{Y}\Omega_0 A_{X})^{-1}$$ coming from~\eqref{eqn:grandfather} has, as usual, degree in the $X_j$ lower by $2 $. This proves the claim in this case.

\medskip

\noindent \emph{Case $u=1,v=1$.} Again splitting the product as $$ ( \Omega_1 A_{X_1} \Omega_{i_1} A_{X_2} \Omega_{i_2} \ldots \Omega_1 A_{X} )\times (\Omega_1 B_{Y}\Omega_0 A_{X}),$$ the first split factor contains the block $ \Omega_1 A_X \Omega_1$ which contributes a factor $(X+1)$ to the trace. The second split factor contributes $XY$. 

In the unsplit product we have from the first occurrence of $A_X$ the block $ \Omega_1 A_X \Omega_1$, which contributes a factor $X+1$, and from the second $A_X$ the block $\Omega_0 A_{X}\Omega_1$, which contributes $X$, again proving our claim.

\medskip

\noindent  \emph{Case $u=0,v=0$.} This can be done by inverting the previous case. Aletrnatively, splitting again as $$ ( \Omega_0 A_{X_1} \Omega_{i_1} A_{X_2} \Omega_{i_2} \ldots \Omega_0 A_{X} )\times (\Omega_1 B_{Y}\Omega_0 A_{X}),$$ the first split factor contains the block $\Omega_0 A_X \Omega_0$, which contributes a factor $X+1$, while the second split factor contributes, as usual, $XY$.

In the unsplit product we have from the first $A_X$ the block $ \Omega_0 A_X \Omega_1$ which contributes $X$, and from the second $A_X$ the block $\Omega_0 A_{X}\Omega_0$ which contributes $X+1$, again proving our claim.

\medskip

\noindent  \emph{Case $u=0,v=1$.} Again split as $$ ( \Omega_0 A_{X_1} \Omega_{i_1} A_{X_2} \Omega_{i_2} \ldots \Omega_1 A_{X} )\times (\Omega_1 B_{Y}\Omega_0 A_{X}).$$ The first split factor, containing $ \Omega_1 A_X \Omega_0$, contributes $X+2$ and the second split factor contributes $XY$. 

In the unsplit product we have from the first $A_X$ the term $\Omega_1 A_X \Omega_1$, which contributes $X+1$, and from the second $A_X$ the term $\Omega_0 A_{X}\Omega_0$, which contributes $X+1$. The induction still works because, up to terms of lower degree, $X(X+2)Y = (X+1)(X+1)Y$.
\end{proof}

\begin{Remark}{\rm
Not all the cases discussed above are realisable as the holonomy representations of simple connected loops $\g$. For example, the cases
$$ v = 1, u = 1, Y = +2 \;\;\;\;\text{and}\;\;\;\; v = 0, u = 0, Y = -2 $$ produce non-simple curves.
}
\end{Remark}

\begin{proof}[Proof of Theorem~\ref{thm:traceformula}]
This follows  from Proposition~\ref{prop:holproduct} on setting $Y_j = \pm2$ in Theorem~\ref{thm:trace2}. (For the final  statement see Lemma~\ref{Parab}.) 
\end{proof}


\end{document}